\documentclass[11pt]{article}

\usepackage{amsmath,amsthm,amssymb,exscale,latexsym,bbm}
\usepackage{mathtools, graphicx, color}
\usepackage{natbib,verbatim}
\usepackage{url} 
\usepackage[a4paper]{geometry}
\usepackage[normalem]{ulem}
\usepackage{enumitem}
\usepackage{tikz}
\usepackage[latin1]{inputenc}
\usepackage[british,english]{babel}
\usepackage[bookmarks=false]{hyperref}

\renewcommand{\Pr}{\mathbb{P}}
\newcommand{\E}{\mathbb{E}}
\newcommand{\N}{\mathbb{N}}
\newcommand{\Z}{\mathbb{Z}}
\newcommand{\R}{\mathbb{R}}
\renewcommand{\AA}{\mathcal{A}}
\newcommand{\CC}{\mathcal{C}}

\newcommand{\joint}{{\mathrm{joint}}}
\newcommand{\indep}{{\mathrm{ind}}}
\newcommand{\Sim}{{\mathrm{sim}}}

\newcommand\unnumberedfootnote[1]{ %
  \let\temp=\thefootnote %
  \renewcommand{\thefootnote}{}%
  \footnote{#1}%
  \let\thefootnote=\temp%
  \addtocounter{footnote}{-1}}

\newcommand{\abs}[1]{\lvert#1\rvert} 
\newcommand{\Abs}[1]{\left\lvert#1\right\rvert}
\newcommand{\norm}[1]{\lVert#1\rVert} 

\newcommand{\ind}[1]{\mathbbm{1}_{\{#1\}}} 
\newcommand{\indset}[1]{\mathbbm{1}_{#1}} 

\newtheorem{theorem}{Theorem}
\newtheorem{proposition}{Proposition}[section]

\newtheorem{lemma}[proposition]{Lemma}

\theoremstyle{definition}

\newtheorem{remark}[proposition]{Remark}

\numberwithin{equation}{section}

\begin{document}

\title{\LARGE Directed random walk on the backbone of an oriented
  percolation cluster}

\author{\sc Matthias Birkner, Ji\v r\'i \v Cern\' y,  \\ \sc Andrej
  Depperschmidt and Nina Gantert
  \\[2ex]
}
\date{16th~June~2013}

\maketitle
\unnumberedfootnote{\emph{AMS 2010 subject classification.} 60K37, 60J10, 82B43, 60K35 }
\unnumberedfootnote{\emph{Keywords and phrases.} Random walk, dynamical
  random environment, oriented percolation, supercritical cluster,
  central limit theorem in random environment}

\begin{abstract}
  \noindent We consider a directed random walk on the backbone of the
  infinite cluster generated by supercritical oriented percolation, or
  equivalently the space-time embedding of the ``ancestral lineage''
  of an individual in the stationary discrete-time contact process. We
  prove a law of large numbers and an annealed central limit theorem
  (i.e., averaged over the realisations of the cluster) using a
  regeneration approach. Furthermore, we obtain a quenched central
  limit theorem (i.e.\ for almost any realisation of the cluster) via
  an analysis of joint renewals of two independent walks on the same
  cluster.
\end{abstract}

\section{Introduction and main results} 
\label{sec:introduction}

In mathematical population genetics it is often important to
understand the ancestral relationship of individuals to deduce
information about the genetic variability in the population, see
\citet{MR2759587}. In spatial population models the ancestry of a
collection of individuals can in general be described by a collection
of coalescing random walks in a random environment. Depending on the
forwards in time evolution of the population such random environments
can be rather complicated.

In the present paper we study the ancestral line, that is, the
evolution of the positions of the parents, of a single individual in
a simple model allowing for locally varying population sizes. More precisely
we consider a discrete-time variant of the contact process: a
$\{0,1\}^{\Z^d}$-valued Markov chain $(\eta_n)_n$ (see below for
precise definitions) where $\eta_{n}(x) = 1$ is interpreted as the
event that the site $x \in \Z^d$ is inhabited by a particle in
generation $n$. We can view this contact process as a ``toy example''
of a spatial stochastic population model with fluctuating population
sizes and local dispersal: Sites $x$ can have \emph{carrying capacity}
$0$ or $1$ in a given generation $n$, and in order for a particle at
site $x$ to be present not only must the corresponding site be {\sl
  open} (i.e.~have the carrying capacity~$1$) but there must also have
been a particle in the neighbourhood of $x$ in the previous generation
$n-1$ who put her \emph{offspring} there. If there was more than one
particle in the neighbourhood of $x$ in generation $n-1$, we think of
randomly assigning one of them to put an offspring to site $x$. Note
that this implicitly models a density-dependent population regulation
because particles in sparsely-populated regions will now have a higher
chance of actually leaving an offspring.

We will let the carrying capacities be i.i.d.\ Bernoulli random
variables, and consider the process $\eta $ in the stationary regime.
In this regime every living particle at generation~$0$,
say, has an infinite line of ancestors. The question of interest is the
distribution of the spatial location of distant-in-time ancestors.

One can of course interpret the discrete-time contact process as a
process describing spread of an infection and interpret the carrying
capacities as \emph{susceptible for the infection} or \emph{immune} if
the Bernoulli random variable at the corresponding site is $1$
respectively $0$. Then our question of interest is the spatial
location of the distant-in-time carriers of the infection from which
the infection propagated to a given individual.

By reversing the time direction, the problem has the following
equivalent description. We consider a simple directed random walk on
the ``backbone'' of the infinite cluster of the oriented percolation
on $\mathbb Z^d\times \mathbb Z$. The backbone of the cluster, we
denote it by $\mathcal C$, is a collection of all sites in $\mathbb
Z^d\times \mathbb Z$ that are connected to infinity by an open path.
The ``time-slices'' of the cluster $\mathcal{C}$ can be seen to be
equal in distribution to the time-reversal of the (non-trivial)
stationary discrete-time contact process $(\eta_n)_n$, and the
directed walk on $\mathcal{C}$ can be interpreted as the spatial
embedding of the ancestral lineage of one individual drawn from the
equilibrium population. The question posed in the previous paragraph
thus amounts to understanding the long time behaviour of this random
walk.

In this formulation, the model is of independent interest in the context of the
theory of random media: The directed random walk on an oriented percolation
cluster can be viewed as a random walk in a Markovian dynamical random
environment. The investigation of such random walks is an active research area
with a lot of recent progress. The random walk we consider however does not
satisfy the usual independence or mixing conditions that appear in the
literature; see Remark~\ref{rem:relationtoRWDREliterature} below. In fact, in
our case the evolution of the environment as a process in time is rather
complicated.

On the other hand, as a random walk on a random cluster, the model is very
natural. The investigation of random walks on percolation clusters is a very
active research area as well. An important difference to our model is that
usually, the walk can move in all (open) directions, whereas we consider a
\emph{directed} random walk.

We now give a precise definition of the model. Let $\omega\coloneqq \{\omega(x,n):
(x,n)\in \Z^d \times \Z\}$ be a family of independent Bernoulli random variables
(representing the carrying capacities) with parameter $p>0$ on some probability
space $(\Omega, \AA, \Pr)$. We call a site $(x,n)$ \emph{open} if
$\omega(x,n)=1$ and \emph{closed} if $\omega(x,n)=0$. We say that there is an
\emph{open path} from $(y,m)$ to $(x,n)$ for $m \le n$ if there is a sequence
$x_m,\dots, x_n$ such that $x_m=y$, $x_n=x$, $\norm{x_k-x_{k-1}} \le 1$ for
$k=m+1, \dots, n$ and $\omega(x_k,k)=1$ for all $k=m,\dots,n$. In this case we
write $(x,m) \to (y,n)$. Here $\norm{\cdot}$ denotes the $\sup$-norm.

Given a set $A \subset \Z^d$ we define the \emph{discrete time contact process}
$(\eta_n^A)_{n \ge m}$ starting at time $m \in \Z$ from the set $A$ as
\begin{align*}
  \eta_m^A (y) & =\indset{A}(y), \; y \in \Z^d, \\
  \intertext{and for $n \ge m$}
  \eta_{n+1}^A(x) & =
  \begin{cases}
    1 & \text{if $\omega(x,n+1)=1$ and $\eta_n^A(y)=1$ for some $y \in \Z^d$
      with $\norm{x-y} \le 1$}, \\
    0 & \text{otherwise}.
  \end{cases}
\end{align*}
In other words, $\eta_n^A(y)=1$ if and only if there is an open path from
$(x,m)$ to $(y,n)$ for some $x\in A$ (where we use in this definition the
convention that $\omega(x,m)=\indset{A}(x)$ while for $k>m$ the $\omega(x,k)$
are i.i.d.\ Bernoulli as above). Often we will identify the configuration
$\eta_n^A$ with the set $\{x \in \Z^d : \eta_n^A(x)=1\}$. Taking $m=0$, we set
\begin{align} \label{def:tauA}
  \tau^{A}\coloneqq  \inf\{n \ge 0: \eta_n^A=\emptyset\},
\end{align}
and in the case $A=\{\boldsymbol 0\}$ we write $\tau^{\boldsymbol 0}$.

It is well known \citep[see e.g.\/][Theorem~1]{GrimmetHiemer:02} that
there is a critical value $p_c \in (0,1)$ such that
$\Pr(\tau^{\boldsymbol 0} =\infty )=0$ for $p\le p_c$ and
$\Pr(\tau^{\mathbf 0} =\infty )>0$ for $p> p_c$. In the following
we consider only the supercritical case $p > p_c$. In this case the
law of $\eta_n^{\Z^d}$ converges weakly to the upper invariant measure
which is the unique non-trivial extremal invariant measure of the
discrete-time contact process. By taking $m \to -\infty$ while keeping
$A=\mathbb Z^d$ one obtains the stationary process
\begin{align}
  \label{eq:5}
  \eta\coloneqq(\eta_n)_{ n \in \Z} \coloneqq (\eta_n^{\Z^d})_{n \in \Z}.
\end{align}

We interpret the process $\eta$ as a population process, where $\eta_n(x)=1$
means that the position $x$ is occupied by an individual in generation $n$. We
are interested in the behaviour of the ``ancestral lines'' of individuals. Note
that because of the discrete time, there can in principle be several individuals
alive in the previous generation that could be ancestors of a given individual
at site $y$, namely all those at some $y'$ with $\norm{y'-y}\leq1$. In that
case, we stipulate that one of these potential ancestors is chosen uniformly at
random to be the actual ancestor, independently of everything else in the model.

Due to time stationarity, we can focus on ancestral lines of individuals
living at time $0$. It will be notationally convenient to time-reverse the
stationary process $\eta$ and consider the process $\xi\coloneqq (\xi_n)_{n \in
  \Z}$ defined by $\xi_n(x)=1$ if $(x,n) \to \infty$ (i.e.~there is an infinite
directed open path starting at $(x,n)$) and $\xi_n(x)=0$ otherwise. Note that
indeed $\mathcal{L}((\xi_n)_{n\in\Z})=\mathcal{L}((\eta_{-n})_{n\in\Z})$. We
will from now on consider the forwards evolution of $\xi$ as the ``positive''
time direction. \smallskip

On the event $B_0\coloneqq\{\xi_0(\mathbf{0})=1\}$ there is an infinite path
starting at $(\mathbf 0,0)$. We define the oriented cluster by
\begin{align*}
  \CC\coloneqq \{(x,n) \in \Z^d \times \Z: \xi_n(x)=1\}
\end{align*}
and let
\begin{equation}
  \label{Udef}
  U(x,n) \coloneqq \{(y,n+1): \norm{x-y} \le 1\}
\end{equation}
be the neighbourhood of the site $(x,n)$ in the next generation. On the
event $B_0$ we may define a $\mathbb Z^d$-valued random walk
$X \coloneqq (X_n)_{n \ge 0}$ starting from $X_0 =\mathbf 0$ with transition
probabilities
\begin{align}
  \label{eq:defXdynamics}
  \Pr( X_{n+1}=y|X_{n}=x, \xi) & =
  \begin{cases}
    {|U(x,n) \cap \CC|}^{-1}  & \text{when }(y,n+1) \in  U(x,n) \cap \CC,\\
    0& \text{otherwise.}
  \end{cases}
\end{align}

Note that $(X_n,n)_{n\ge 0}$ is a directed random walk on the
percolation cluster $\CC$, and $X$ can be also viewed as a random walk
in a (dynamical) random environment, where the environment is given by
the process $\xi$. As the environment $\xi$ is the time-reversal of
the stationary Markov process $\eta $, it is itself Markovian and
stationary, the invariant measure being the upper invariant measure of
the discrete-time contact process $\eta$. While the evolution of
$\eta$ is easy to describe forwards in time by local rules, $\eta$ is
not reversible, and the time evolution of its reversal $\xi$ seems
complicated. The transition probabilities for $\xi$ cannot be
described by local functions: For example, when viewed as a function
of $a =(a_x)_{x\in\Z^d} \in \{0,1\}^{\Z^d}$, $f(a) \coloneqq
\Pr(\xi_{n+1}(\boldsymbol{0})=1 | \xi_n=a)$, there will be no finite
set $K \subset \Z^d$ such that $f$ depends only on $(a_x)_{x\in K}$
(this can for example be seen by considering various $a$'s that have a
``sea of $0$'s around the origin''). Presently it is not at all clear
to us how to describe the forwards in time evolution of $\xi$ in a
more tangible way. Note however, that the process $\xi$ does form a
finite-range Markov field in the larger space $\Z^d \times \Z$ because
this is true for $\eta$, but it is unclear at the moment what use we
could make of that fact.

The complicated nature of $\xi $ disallows checking many of the usual conditions
that appear in the literature on random walks in dynamical random environment.
Some of such conditions (like e.g.\ ellipticity) are obviously violated by our
model. To our best knowledge, the random walk in a (dynamic) random environment
that we consider here is not contained in one of the classes studied in the
literature so far; see Remark~\ref{rem:relationtoRWDREliterature} below.
\smallskip

Our first result shows the law of large numbers, and a central limit theorem for
$X$ when averaging over both the randomness of the walk's steps and the
environment $\omega$. We write $P_\omega$ for the conditional law of $\Pr$,
given $\omega$, and $E_\omega $ for the corresponding expectation. With this
notation we have $\Pr( X_{n+1}=y|X_{n}=x, \xi) = P_\omega( X_{n+1}=y|X_{n}=x)$.

\begin{theorem}[LLN \& annealed CLT]
  \label{thm:qLLN-aCLT}
  For any $d \ge 1$ we have
  \begin{align}
    \label{eq:LLN}
    P_\omega\Big( \frac1n X_n \to \mathbf{0} \,\Big) = 1
    \quad \text{for $\Pr(\,\cdot\, | B_0)$-a.a.\ $\omega$},
  \end{align}
  and for any $f \in C_b(\R^d)$
  \begin{align}
    \label{eq:annealedCLT}
    \E \Big[ f\left( X_n/\sqrt{n}\, \right) \, \Big| \, B_0 \Big]
    \xrightarrow{n \to
      \infty}   \Phi(f),
  \end{align}
  where $\Phi(f)\coloneqq\int f(x)\, \Phi(dx)$ with $\Phi$ a non-trivial
  centred isotropic $d$-dimensional normal law.
\end{theorem}
We prove this theorem by exhibiting a regeneration structure for $X$ and
$\xi $, and then showing that the second moments of temporary and spatial
increments of $X$ at regeneration times are finite (in fact we will prove
existence of exponential moments).
\begin{remark}
\label{rem:sigma.p}
The covariance matrix of $\Phi$ in \eqref{eq:annealedCLT} is
  $\sigma^2$ times the $d$-dimensional identity matrix. It is clear from the
  construction below (see Subsection~\ref{subsect:regentimes}) that
\begin{equation}
  \sigma^2 = \sigma^2(p) = \frac{\E\big[ Y_{1,1}^2 \big]}{\E[\tau_1]} \in (0, \infty)
\end{equation}
where $\tau_1$ is the first regeneration time (see \eqref{deftauiYi})
of the random walk $X$ and $Y_{1,1}$ is the first coordinate of
$X_{\tau_1}$, the position of the random walk at this regeneration
time. A simple truncation and coupling argument shows that $p \mapsto
\sigma^2(p)$ is continuous on $(p_c,1]$; see
Remark~\ref{rem:sigma-cont}. The behaviour of $\sigma^2(p)$ as $p
\downarrow p_c$ appears to be an interesting problem that merits
further research.
\end{remark}

It is natural to study also two (or even more) walkers on the same cluster. On
the one hand, this allows to obtain information on the long-time behaviour in a
multi-type situation. \citet{Neuhauser:1992} and more recently
\citet{Valesin:2010} employed this for the (continuous-time) contact process. It
is also very natural from the modelling of ancestral lineages point of view,
where it corresponds to jointly describing the space-time embedding of the
ancestry of a sample of size two (or more) individuals when the walks start from
different sites. On the other hand, good control of the behaviour of two or more
``replicas'' of $X$ on the same cluster allows us to strengthen the annealed CLT
\eqref{eq:annealedCLT} to the quenched version.

\begin{theorem}[Quenched CLT]
  \label{thm:quenchedCLT}
  For any $d\ge 1$ and $f \in C_b(\R^d)$
  \begin{equation}
    \label{eq:quenchedCLT}
    E_\omega\Big[f\left(X_n/\sqrt{n} \, \right) \,\Big] \xrightarrow{n \to
      \infty} \Phi(f)  \quad    \text{ for  $\Pr(\,\cdot\, |
        B_0)$-a.a.~$\omega$,}
  \end{equation}
  where $\Phi$ is the same non-trivial centred isotropic
  $d$-dimensional normal law as in \eqref{eq:annealedCLT}.
\end{theorem}

Let us conclude the introduction by mentioning some generalisations of
the random walk that we consider here.

\begin{remark}[More general neighbourhoods]
  \label{rem:gen-neighb}
  For simplicity, we defined $U(x,n)$ as in \eqref{Udef}, but the
  proofs go through for any finite, symmetric neighbourhood (by
  ``symmetric'' we mean that $y \in U(x,n)$ if and only if $-y \in
  U(x,n)$). In this case the resulting law $\Phi$ will in general not
  be isotropic, see the end of the proof of Theorem
  \ref{thm:qLLN-aCLT}.

  Note also, that for sake of clarity, all figures in this paper are
  drawn with $U(x,n)$ set to $\{(x+1,n+1),(x-1,n+1)\}$ for $d=1$.
\end{remark}

\begin{remark}[Functional central limit theorem]
  \label{rem:inv-principle}
  For the random walk $X$ one can also deduce corresponding annealed
  and quenched functional limit theorems; see also
  Remark~\ref{rem:inv-principle-2}.
\end{remark}

\begin{remark}[Contact process with fluctuating population size]
  \label{rem:gen-cp}
  Let $K(x,n)$, $(x,n) \in \Z^d \times \Z$ be i.i.d.\ $\N=\{1,2, \dots\}$-valued
  random variables and let us define the \emph{discrete time contact process
    with fluctuating population size}, $\widehat \eta \coloneqq(\widehat
  \eta_n)_{n \in \Z}$, by
  \begin{align}
    \label{eq:8}
    \widehat \eta_n(x) \coloneqq \eta_n(x) K(x,n),
  \end{align}
  and its reversal $\widehat \xi \coloneqq ( \widehat \xi_n)_{n \in \Z}$ by
  \begin{align}
    \label{eq:10}
    \widehat \xi_n(x) \coloneqq \xi_n(x) K(x,n).
  \end{align}
  One can interpret $K(x,n)$ as a random ``carrying capacity'' of the site
  $(x,n)$. Now conditioned on $\widehat\xi_0(\mathbf{0}) \ge 1$ the ancestral random
  walk, we call it $X$ as before, can be defined by $X_0=\boldsymbol 0$ and
  \begin{equation}
    \label{eq:11}
    \Pr\bigl( X_{n+1}=y|X_{n}=x, \widehat \xi\, \bigr) =
    \begin{cases}
      \dfrac{\widehat \xi_{n+1}(y)} {\sum_{(y',n+1) \in U(x,n)}
        \widehat\xi_{n+1}(y')} \,
      \quad &\text{if }(y,n+1) \in  U(x,n),\\[4mm]
      0&\text{otherwise.}
    \end{cases}
  \end{equation}
  For such random walks in random environments our arguments can also be adapted
  and the same results as above can be obtained; see also
  Remark~\ref{rem:alt-constr-fluc-cp}.
\end{remark}

\begin{remark}[Relation to other approaches to RWDRE in the literature]
\label{rem:relationtoRWDREliterature}
Random walks in dynamic random environments (RWDRE) are currently an active area
of research (they can of course by explicitly including the ``time'' coordinate
in principle be expressed in the context of random walk in random (non-dynamic)
environments, but the often more complicated structure of the law of the
environment does make them somewhat special inside this general class). To the
best of our knowledge, the walk \eqref{eq:defXdynamics} and our results in
Theorems~\ref{thm:qLLN-aCLT} and \ref{thm:quenchedCLT} are not covered by
approaches in the literature. Here is a list of corresponding
results so far together with a very brief discussion:
\begin{itemize}
\item \citet{DolgopyatKellerLiverani.AOP36} obtain a CLT under ``abstract''
  conditions on the environment process, that appear very hard to verify
  explicitly for $\xi$, in particular the absolute continuity condition for
  $\xi$ viewed relative to the walk w.r.t.\ the a priori law on $\xi$.
\item \citet{MR2748407} consider environments that are ``refreshed in each
  step'' (i.e.\ time-slices are i.i.d.), this does not apply to $(\xi_n)$.
\item Individual coordinates $(\xi_n(x))_{n\in\N}$ with
$x \in \Z^d$ fixed are not (independent) Markov chains,
in contrast to the set-up in \citet{MR2507753}.
\item $(\xi_n)$ does not fulfil the required uniform coupling conditions
  employed in \citet{RedigVoellering.arxiv1106.4181}.
\item $(\xi_n)$ does not fulfil the cone mixing condition considered
  in \citet{MR2786643}.
  Intuitively, this stems from the fact that the supercritical
  contact process has two extremal invariant distributions, the
  upper invariant measure (which we consider here) and the trivial measure
  (which concentrates on $\xi \equiv 0$). Thus, at an arbitrarily large
  level $n$ an atypically high density of zeros around the origin
  can be achieved by conditioning on large enough islands of zeros below at
  level $0$, an event with positive probability.
\item \citet{denHollanderdosSantosSidoravicius2012} weaken the
  cone-mixing condition from \citet{MR2786643} to a conditional
  cone-mixing condition and obtain a LLN for a class of
  (continuous-time) random walks in dynamic random environments with
  this (and some further technical) assumptions. Further research is
  required to investigate whether a similar condition can be
  established for $\xi$ but note that presently, the approach in
  \citet{denHollanderdosSantosSidoravicius2012} does not yield a CLT.
\end{itemize}
\end{remark}

The rest of the paper is organised as follows. In Section~\ref{sec:reg-struct},
we define the regeneration structure for a single walk. To define the
regeneration times, we first give an alternative construction of the walk, using
some ``external randomness''. Theorem~\ref{thm:qLLN-aCLT} then follows by
standard arguments, once we show that the regeneration times have finite
exponential moments, see Lemma~\ref{lem:regenstructure}. In
Section~\ref{sec:joint-dynamics}, we define a joint regeneration structure for
two walks on the same cluster, and we compare their joint distribution with the
distribution of two walks in two independent copies of the environment, which
allows us to prove Theorem~\ref{thm:quenchedCLT}. Finally, in
Appendix~\ref{sect:auxproofs} we prove a ``folklore'' result (that we could not
find in this form in the literature), namely that the height of a finite cluster
in the oriented percolation has, for all supercritical parameters, an
exponential tail.

\section{Regeneration structure for a single random walk}
\label{sec:reg-struct}

In this section we describe and study a regeneration structure of the random
walk $X$ conditioned on the event $B_0$. We adapt arguments from
\citet{Kuczek:1989} and \citet{Neuhauser:1992} and show that the regeneration
times have some exponential moments and consequently finite second moments. From
that Theorem~\ref{thm:qLLN-aCLT} follows by standard arguments. The
corresponding proof is given after Lemma~\ref{lem:regenstructure}.

The construction of the regeneration structure is lengthy but not difficult. Its
goal is to build a trajectory of $X$ using rules that are ``local'', i.e.~which
use only local $\omega $'s (and some additional local randomness), but not $\xi
$'s, as to know $\xi $'s we need to know the ``whole future'' of the environment
$\omega $. Of course, this is not possible in general, but the regeneration
times which we will construct are exactly the times when the locally constructed
trajectory coincides with the trajectory of $X$. We note that a similar but
non-randomised construction was used in \citet{SarkarSun2013} to
analyse the collection of rightmost paths in a directed percolation cluster.

\subsection{Local construction of random walk on $\mathcal C$}
\label{ss:localconst}

We will need some additional randomness for the construction: For every
$(x,n) \in \Z^d \times \Z$ let $\widetilde\omega{(x,n)}$
be a uniformly chosen permutation of $U(x,n)$, independently distributed for all
$(x,n)$'s, defined also on the probability space $(\Omega ,\mathcal A, \mathbb
P)$. We denote the whole family of these permutations by $\widetilde\omega$.

For every $(x,n)\in \mathbb Z^d\times \mathbb Z$ let
$\ell(x,n)=\ell_\infty(x,n)$ be the length of the longest directed
open path starting at $(x,n)$; we set $\ell(x,n)=-1$ when $(x,n)$ is
closed. (Recall that a path $(x_0,n), (x_1,n+1), \dots, (x_k,n+k)$ of
length $k$ with $\norm{x_i-x_{i-1}} \leq 1$ is open if
$\omega(x_0,n)=\omega(x_1,n+1)=\cdots= \omega(x_k,n+k)=1$.) For every
$k\in \{0,1,\dots\}$ let $\ell_k(x,n)\coloneqq\ell(x,n)\wedge k$ be
the length of the longest directed open path of length at most $k$
starting from $(x,n)$. Observe that $\ell_k(x,n)$ is measurable with
respect to the $\sigma $-algebra $\mathcal G_{n}^{n+k+1}$, where
\begin{align}
  \label{eq:defGnnk}
  \mathcal{G}_{n}^{m} \coloneqq
  \sigma\big(\omega(y,i), \widetilde\omega{(y,i)}  :
    y \in \Z^d, n \leq i < m \big), \quad n < m.
\end{align}
For $k\in\{0,\dots,\infty\}$, we define $M_k(x,n)\subseteq U(x,n)$ to be the set of
sites which maximise $\ell_k$ over $U(x,n)$, i.e.\
\begin{equation}
  M_k(x,n) \coloneqq \Big\{y\in U(x,n):\ell_k(y)=\max_{z\in U(x,n)}
  \ell_k(z)\Big\},
\end{equation}
and for convenience we set $M_{-1}(x,n)=U(x,n)$.
Observe that we have
\begin{align}
  \label{e:M0}
  M_0(x,n)&=\{y\in U(x,n): y\text{ is open}\},\\
  \label{e:Minfty}
  M_\infty(x,n)&= U(x,n)\cap \mathcal C, \\
  \label{e:Mincl}
  M_k(x,n)&\supseteq M_{k+1}(x,n), \qquad k\ge -1.
\end{align}
Let $m_k(x,n)\in M_k(x,n)$ be the element of $M_k(x,n)$ that appears as the
first in the permutation $\widetilde\omega{(x,n)}$.

Given $(x,n)$, $k$, $\omega $ and $\widetilde\omega$, we define a path
$\gamma_k = \gamma_k^{(x,n)}$ of length $k$ via
\begin{equation}
  \begin{split}
    \label{e:gamma}
    \gamma_k(0) &= (x,n),\\
    \gamma_k(j+1) &= m_{k-j-2}(\gamma_k(j)),\qquad j=0,\dots,k-1.
  \end{split}
\end{equation}
In words, at every step, $\gamma_k$ checks the neighbours of its
present position and picks randomly (using the random permutation
$\widetilde \omega $) one of those where it can go further on open
sites, but inspecting only the state of sites in the time-layers
$\{n,\dots,n+k-1\}$. Consequently, the construction of
$\gamma_k^{(x,n)}$ is measurable with respect to the $\sigma $-algebra
$\mathcal G_n^{n+k}$ (recall \eqref{eq:defGnnk}). See
Figure~\ref{fig:gammapaths} for an illustration. \smallskip

\begin{figure}
 \centering
 \begin{tikzpicture}[xscale=0.65,yscale=0.65]
   \draw (0,-0.5) node {\footnotesize $(x,n)$};
   \draw (2.1,0.65) node {\footnotesize $\widetilde \omega_{(x,n)}(1)$};
   \draw (-2.05,0.65) node {\footnotesize $\widetilde \omega_{(x,n)}(2)$};
   \filldraw[fill=black] (0,0) circle (3pt)
                         (-1,1) circle (3pt)
                         (1,1) circle (3pt)
                         (2,2) circle (3pt)
                         (-2,2) circle (3pt)
                         (-3,3) circle (3pt)
                         (-1,3) circle (3pt);
   \filldraw[fill=white] (0,2) circle (3pt)
                         (3,3) circle (3pt)
                         (1,3) circle (3pt);

   \draw[-stealth,thick] (0.1,0.1) -- (0.9,0.9);
   \draw[-stealth,thick] (-0.9,1.1) -- (-0.1,1.9);
   \draw[-stealth,thick] (0.9,1.1) -- (0.1,1.9);
   \draw[-stealth,thick] (-2.1,2.1) -- (-2.9,2.9);
   \draw[-stealth,thick] (0.1,2.1) -- (0.9,2.9);
   \draw[-stealth,thick] (2.1,2.1) -- (2.9,2.9);

   \draw[-stealth,thick] (3.1,3.1) -- (3.9,3.9);
   \draw[-stealth,thick] (-2.9,3.1) -- (-2.1,3.9);
   \draw[-stealth,thick] (-0.9,3.1) -- (-0.1,3.9);
   \draw[-stealth,thick] (0.9,3.1) -- (0.1,3.9);

   \draw[-stealth,thick,dotted] (-0.1,0.1) -- (-0.9,0.9);
   \draw[-stealth,thick,dotted] (-1.1,1.1) -- (-1.9,1.9);
   \draw[-stealth,thick,dotted] (-1.9,2.1) -- (-1.1,2.9);
   \draw[-stealth,thick,dotted] (-0.1,2.1) -- (-0.9,2.9);
   \draw[-stealth,thick,dotted] (1.1,1.1) -- (1.9,1.9);
   \draw[-stealth,thick,dotted] (1.9,2.1) -- (1.1,2.9);
   \draw[-stealth,thick,dotted] (-3.1,3.1) -- (-3.9,3.9);
   \draw[-stealth,thick,dotted] (-1.1,3.1) -- (-1.9,3.9);
   \draw[-stealth,thick,dotted] (1.1,3.1) -- (1.9,3.9);
   \draw[-stealth,thick,dotted] (2.9,3.1) -- (2.1,3.9);

   \draw[thick] (5,0) -- (6,1);
   \filldraw[fill=black] (5,0) circle (1.5pt) (6,1) circle (1.5pt);
   \draw (5.5,-0.4) node {\footnotesize $k=1$};

   \draw[thick] (8,0) -- (9,1) -- (8,2);
   \filldraw[fill=black] (8,0) circle (1.5pt) (9,1) circle (1.5pt) (8,2) circle (1.5pt);
   \draw (8.5,-0.4) node {\footnotesize $k=2$};

   \draw[thick] (10,0) -- (11,1) -- (12,2) -- (13,3) ;
   \filldraw[fill=black] (10,0) circle (1.5pt) (11,1) circle (1.5pt) (12,2) circle
   (1.5pt) (13,3) circle (1.5pt);
   \draw (11.5,-0.4) node {\footnotesize $k=3$};

   \draw[thick] (17,0) -- (16,1) -- (15,2) -- (14,3) -- (15,4) ;
   \filldraw[fill=black] (17,0) circle (1.5pt) (16,1) circle (1.5pt) (15,2) circle
   (1.5pt) (14,3) circle (1.5pt) (15,4) circle (1.5pt);
   \draw (15,-0.4) node {\footnotesize $k=4$};

 \end{tikzpicture}
 \caption{The paths $\gamma_k^{(x,n)}$ from \eqref{e:gamma} based on
   $\omega$'s and $\widetilde\omega$'s. Black and white circles represent
   open sites, i.e.\ $\omega(\text{site})=1$, and closed sites, i.e.\
   $\omega(\text{site})=0$, respectively. Solid arrows from a site point
   to $\widetilde\omega_{(\text{site})}(1)$ and dotted to
   $\widetilde\omega_{(\text{site})}(2)$. On the right the sequence of
   paths $\gamma_k^{(x,n)}(\cdot)$ for $k=1,2,3,4$ is shown. For sake of
   clarity, we used $U(x,n)=\{(x+1,n+1),(x-1,n+1)\}$,
   cf.~Remark~\ref{rem:gen-neighb}.}
 \label{fig:gammapaths}
\end{figure}
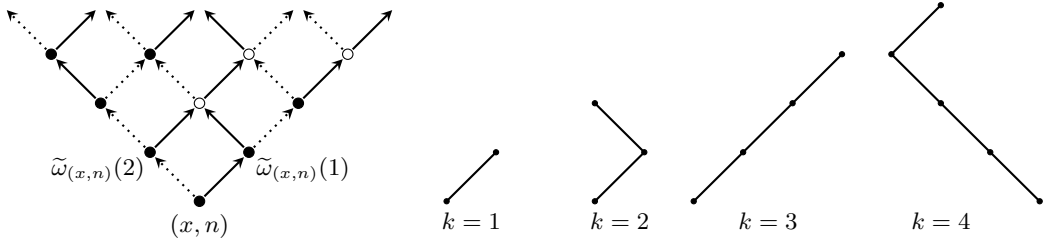

The paths $\gamma_k^{(x,n)}$ have the following properties.
\begin{lemma}
  \label{l:gamma}
  Assume that $\omega $ and $(x,n)\in \mathcal C$ are fixed.
  \begin{itemize}
    \item[(a)] The law of
    $(\gamma_\infty^{(x,n)}(j))_{j\ge 0}$ is the same as the law of the random
    walk $(X_j,n+j)_{j\ge 0}$ on $\mathcal C$ started from $(x,n)$.
  \end{itemize}
  If in addition $\widetilde\omega$ is fixed, then
  \begin{itemize}
    \item[(b)] $\omega(\gamma_k(m))=1$ for all $0\le m < k$.
    \item[(c)] (stability in $k$)
    If the end point of $\gamma_k$ is open,
    i.e.~$\omega (\gamma_k(k))=1$, then the path
    $\gamma_{k+1}$ restricted to the first $k$ steps equals
    $\gamma_k$.

    \item[(d)] (fixation on $\mathcal C$)
    Assume that $\gamma_k(j)\in \mathcal C$ for some $k\ge 0$, $j\le k$.
    Then, $\gamma_m(j)=\gamma_k(j)$ for all $m>k$.

    \item[(e)] (exploration of finite branches) If
    $\gamma_k(k-1)\in \mathcal C $ and $\gamma_k(k)\notin\mathcal  C$ for
    some $k$, then $\gamma_j(k)=\gamma_k(k)$ for all
    $k\le j \le k+\ell(\gamma_k(k))+1$ and
    $\gamma_{k+\ell(\gamma_k(k))+2}(k)\neq \gamma_k(k)$.
  \end{itemize}
\end{lemma}

\begin{proof}
  Claim (a) follows directly from \eqref{e:Minfty}, the fact that
  $m_\infty(\cdot)$ is a uniformly chosen element of
  $M_\infty(\cdot)$, and the definition of the path $\gamma_\infty$.
  For claim (b), it is sufficient to observe that when $(x,n)\in
  \mathcal C$, there is an open path of length $k-1$ starting at
  $(x,n)$ which $\gamma_{k}$ will follow. For (c), if $\gamma_k(k)$ is
  open, then $m_{k-j-2}(\gamma_k(j)) \in M_{k-j-1}(\gamma_k(j))$ for
  every $0\le j < k$, and thus, using the inclusion \eqref{e:Mincl},
  $m_{k-j-1}(\gamma_k(j))=m_{k-j-2}(\gamma_k(j))$, for $0\le j<k$. For
  (d), if $\gamma_k(j)$ is on $\mathcal C$, then $\gamma_k(j)\in
  M_m(\gamma_k(j-1))$ for every $m>k$ by \eqref{e:Minfty}, and thus
  $\gamma_k(j)=\gamma_m(j)$. Finally, (e) follows by observing that
  when $\gamma_k(k) = m_{-1}(\gamma_k(k-1))\notin \mathcal C$, that is
  $\ell(\gamma_k(k))<\infty$, then $\gamma_k(k)=m_j(\gamma_k(k-1))\in
  M_j(\gamma_k(k-1))$ for all $0\le j \le \ell(\gamma_k(k))$ but
  $\gamma_k(k)\notin M_{\ell(\gamma_k(k))+1}(\gamma_k(k-1))$.
\end{proof}

\begin{remark}
  \label{r:coupling}
  Lemma~\ref{l:gamma}(a) allows to couple the random variables
  $\omega ,\widetilde\omega$ with the random walk $(X_k,k)$ started from
  $(\boldsymbol 0,0)$ by setting
  \begin{equation}
    \label{eq:walkpathcoupled}
    (X_k,k)= \gamma_\infty^{(\boldsymbol 0,0)}(k)
    =\lim_{j\to\infty} \gamma_j^{(\boldsymbol 0,0)}(k).
  \end{equation}
  (Note that the limit on the right-hand side exists by Lemma~\ref{l:gamma}(d).)
  From now on, we will assume that this coupling is in place.
\end{remark}

\begin{remark}
  1.\ This construction can \emph{a priori} be used to
  extend the definition of the random walk $X$ for starting points
  that are not on the infinite cluster $\mathcal C$. It is sufficient
  to use similar arguments as in the previous lemma to show that for
  every $(x,n)$, $\omega $, and~$\widetilde\omega$,
  $(X_k,n+k)=\lim_{j\to\infty} \gamma^{(x,n)}_j(k)$ exists a.s.~for
  every $k$ and is a directed path, which remains on $\mathcal C $
  once it hits it. Actually, in this way we obtain a coalescing flow
  on $\mathbb Z^d\times\mathbb Z$. \\[0.5ex]
  2.\ In analogy with the construction in
  \citet[Section~2]{Neuhauser:1992} we can think of the path
  $\gamma_k$ defined in \eqref{e:gamma} as leading to $\gamma_k(k)$,
  the \emph{first} potential ancestor $k$ generations ago of the
  particle at $(x,n)$. The construction of $\gamma_k$ can easily be
  extended to a random ordering of all paths of length $k$, thus
  yielding an ordered sequence of all potential
  ancestors. \\[0.5ex]
  We will not need these extensions in the present paper.
\end{remark}

\begin{remark}[The construction in the case of fluctuating local
    population size]%
  \label{rem:alt-constr-fluc-cp}%
  \hfill
  In the case of fluctuating local population size as in
  Remark~\ref{rem:gen-cp}, the same construction can be performed. To this
  end it is only necessary to replace the uniform random permutation
  $\widetilde\omega{(x,n)}$ by a ``weighted'' random
  permutation with distribution
  \begin{align}
    \label{eq:15}
    \Pr\bigl(\widetilde\omega{(x,n)} =(y_{1}, \dots, y_{|U(x,n)|}  ) | K  \bigr) =
    \frac{1}{Z(x,n)} \prod_{\ell =1}^{|U(x,n)|} K(y_{\ell}),
  \end{align}
  where $(y_{1}, \dots, y_{|U(x,n)|})$ run over all permutations of $U(x,n)$
  and $Z (x,n)$ is the normalisation factor.
\end{remark}

\subsection{Regeneration times}
\label{subsect:regentimes}

We can now introduce the regeneration times which will be used to show
all main results of the present paper. We consider the random walk
$(X_n,n)$ started at $(\boldsymbol 0,0)$ as defined in
\eqref{eq:walkpathcoupled} and write $\gamma_k$ for
$\gamma_k^{(\boldsymbol 0,0)}$. We define a sequence $T_j$,
$j\ge 0$, by
\begin{align}
  \label{eq:reg.1}
  T_0\coloneqq 0 \quad \text{and}\quad
  T_{j}\coloneqq\inf\left\{k> T_{j-1}: \xi(\gamma_k(k))=1\right\},\quad j\ge 1.
\end{align}
(Here and later we use the notation $\xi (y) \coloneqq \xi_n(x)$ when
  $y=(x,n)\in\mathbb Z^d\times \mathbb Z$.)

At times $T_j$ the local construction of the path finds a ``real ancestor''
of $(\boldsymbol 0,0)$ in the sense that for any $m>T_j$,
$\gamma_m(T_j)=\gamma_{T_j}(T_j)$, by
Lemma~\ref{l:gamma}(d). Therefore, the local construction ``discovers the
trajectory of $X$ up to time $T_j$''. More precisely we know that
$(X_m,m)=\gamma_{T_j}(m)$
for all $0\le m\le T_j$, cf.~Lemma~\ref{l:gamma}(a) and
Remark~\ref{r:coupling}.

For $i=1,2,\dots$ we set
\begin{align}
\label{deftauiYi}
\tau_i\coloneqq T_i-T_{i-1} \qquad \text{and} \qquad
Y_i\coloneqq X_{T_i}-X_{T_{i-1}}.
\end{align}
The strong law of large numbers as well as the
(averaged) central limit theorem are consequences of the following lemma.

\begin{lemma}[Independence and exponential tails for regeneration increments]
  \label{lem:regenstructure}
  \hfill Condi\-tioned on $B_0$ the sequence $\bigl((Y_i,\tau_i)\bigr)_{i \ge 1}$ is
  i.i.d.\ and $Y_{1}$ is symmetrically distributed. Furthermore, there exist
  constants $C,c\in(0,\infty)$, such that
  \begin{align} \label{eq:bounds}
    \Pr(\norm{Y_{1}} > n | B_0) \le Ce^{-c n} \quad \text{and} \quad
    \Pr(\tau_1 >n|B_0) \le Ce^{-c n}.
  \end{align}
\end{lemma}

\begin{proof}[Proof of Theorem~\ref{thm:qLLN-aCLT}]
  By symmetry and \eqref{eq:bounds}, we have $\E[Y_{1}| B_0]= \mathbf{0}$.
  Non-triviality of $\Phi$ follows since $T_1$ and $Y_{1}$ are not
  deterministic multiples of each other and $Y_{1}$ is not
  concentrated on a subspace which follows from the fact that
  $\Pr(Y_{1}=x, \tau_1=n | B_0)>0$ for all $x \in \Z^d$ and $n \geq
  \norm{x}$. To see this we observe that the configuration of the
  $\omega$'s in a space-time box of side-length $n$ around the origin
  consisting only of closed sites except for the origin itself and two
  disjoint ``rays'' of open sites, the first connecting
  $(\mathbf{0},0)$ to $(-x,n-1)$ and ending there, the second
  connecting $(\mathbf{0},0)$ up to $(x,n)$, has positive probability.

  The rest of the proof is standard, see e.g.\ the proof of Corollary~1 in
  \citet{Kuczek:1989}, or the proof of Theorem 4.1. in \citet{MR1763302}.

  Note that since the ``basic neighbourhood'' $U(\cdot,\cdot)$ is
  symmetric, $\Phi$ is isotropic, i.e.\ its covariance matrix
  $(\Sigma_{ij})$ is a multiple of the $d$-dimensional identity
  matrix: Because $\Phi$ is invariant under permutation of
  coordinates, we must have $\Sigma_{ii}=\Sigma_{jj}=\sigma^2 \in
  (0,\infty)$, $\Sigma_{ij} = s \in \R$ for all $1 \leq i \neq j \leq
  d$. Furthermore, the law $\Phi$ then also inherits invariance under
  sign flips of individual coordinates, hence we must have $s=-s=0$.
\end{proof}

\subsection{Proof of Lemma~\ref{lem:regenstructure}}

Symmetry of the law of $Y_{1}$ follows from the symmetry of the
construction. Note that $\norm{Y_{1}} \le \tau_1$ a.s. Thus, the first
bound in \eqref{eq:bounds} follows from the second, which we prove
now.

The value of $T_1$ can be obtained by gradually constructing
$\gamma_k$ and checking whether $\xi (\gamma_k(k))= 1$ for
$k=1,2,\dots$. We abbreviate $\widetilde\Pr(\cdot) := \Pr(\cdot |
B_0)$. Let $\sigma$ be a $(\mathcal G_0^k)_{k\in\N}$-stopping time,
denote the $\sigma$-past by $\mathcal G_0^\sigma$, let $W$ be a
$\mathcal G_0^\sigma$-measurable $\Z^d$-valued random variable. An
application of the FKG inequality yields
\begin{align}
\label{eq:FKG-cond}
\widetilde\Pr\big( \xi((W,\sigma)) = 1 \, \big| \, \mathcal G_0^\sigma \big)
\geq \Pr(B_0),
\end{align}
as follows:
For any $n\in \N$, $B_0$ can be written as the finite
disjoint union
\[
B_0 = \bigcup_{S \subset B_n(\mathbf{0})}
\bigg( \big\{ (\boldsymbol 0,0) \Rightarrow S \times \{n\} \big\}
\cap \Big(
\bigcup_{y \in S} \{ (y,n) \to \infty \} \Big) \bigg)
\]
where $B_n(\mathbf{0})$ denotes the $\norm{\cdot}$-ball of radius $n$ around $0$
in $\Z^d$ and $\big\{ (\boldsymbol 0,0) \Rightarrow S \times \{n\} \big\} \in
\mathcal{G}_0^n$ denotes the event that the set of sites $y \in \Z^d$
with the property that $(y,n)$ can be reached from $(\boldsymbol 0,0)$ via a
directed nearest neighbour path whose steps begin on open sites
equals exactly $S$.

Now pick $A \in \mathcal G_0^\sigma$. We have
\begin{align*}
& \Pr \big( \{ \xi((W,\sigma)) = 1 \} \cap A \cap B_0 \big)
= \sum_{w, n}
\Pr\big( \{ \sigma=n, W=w \} \cap A \cap \{ (w,n) \to \infty \} \cap B_0 \big)
\\
& = \sum_{w, n} \sum_{S \subset B_n(\mathbf{0})}
\Pr\Big( \{ \sigma=n, W=w \} \cap A \cap
\big\{ (\boldsymbol 0,0) \Rightarrow S \times \{n\} \big\} \\[-3ex]
& \hspace{17em} \cap
\{ (w,n) \to \infty \} \cap \bigcup\nolimits_{y \in S} \{ (y,n) \to \infty \} \Big) \\
& = \sum_{w, n} \sum_{S \subset B_n(\mathbf{0})}
\Pr\Big( \{ \sigma=n, W=w \} \cap A \cap
\big\{ (\boldsymbol 0,0) \Rightarrow S \times \{n\} \big\} \Big) \\[-3ex]
& \hspace{16em} \times \Pr\Big(
\{ (w,n) \to \infty \} \cap \bigcup\nolimits_{y \in S} \{ (y,n) \to \infty \} \Big) \\
& \geq \sum_{w, n} \sum_{S \subset B_n(\mathbf{0})}
\Pr\Big( \{ \sigma=n, W=w \} \cap A \cap
\big\{ (\boldsymbol 0,0) \Rightarrow S \times \{n\} \big\} \Big) \\[-3ex]
& \hspace{16em} \times
 \Pr\Big(\bigcup\nolimits_{y \in S} \{ (y,n) \to \infty \} \Big)
 \times \Pr\Big( \{ (w,n) \to \infty \} \Big) \\
& = \sum_{w, n} \sum_{S \subset B_n(\mathbf{0})}
\Pr\Big( \{ \sigma=n, W=w \} \cap A \cap
\big\{ (\boldsymbol 0,0) \Rightarrow S \times \{n\} \big\} \cap
\bigcup_{y \in S} \{ (y,n) \to \infty \} \Big) \Pr(B_0) \\
& = \Pr\big( A \cap B_0 \big) \Pr(B_0)
\end{align*}
where we used the FKG inequality in the fourth line and the fact that
$\mathcal G_0^n$ and $\mathcal G_n^\infty$ are independent in the
third and the fifth lines. Since $A \in \mathcal G_0^\sigma$ is arbitrary, this
proves \eqref{eq:FKG-cond}.
\smallskip

Applying \eqref{eq:FKG-cond} with $\sigma=\sigma_0:=1$ and
$W=\gamma_{\sigma_0}(\sigma_0)$ yields
\begin{equation}
  \label{e:fkg}
  \widetilde\Pr\left(T_1=1\right)=
  \widetilde\Pr\left(T_1=1 | \mathcal{G}_0^1\right)
  = \widetilde\Pr\left(\xi (\gamma_1(1))=1 | \mathcal{G}_0^1\right)
  \geq \Pr(B_0) .
\end{equation}
When $\gamma_1(1)\notin \mathcal C$, we should wait for the local
construction to discover this fact, i.e.\ for the finite directed
cluster starting at $\gamma_1(1)$ to die out. More precisely, on the
event $\mathcal B_1=\{\gamma_1(1)\notin \mathcal C\}$,
$\ell(\gamma_1(1))$, the length of the longest directed open path
starting at $\gamma_1(1)$ is finite and the local construction
discovers this fact at time $\ell(\gamma_1(1)) + \sigma_0 +1$ when the
longest paths gets stuck, i.e.\ when it runs into a ``dead end''
produced by closed sites. Thus, $\sigma_1$, defined by
$\sigma_1=\ell(\gamma_1(1))+2+\sigma_0$ is a stopping time w.r.t.\ the
filtration $(\mathcal{G}_0^m)_{m=1,2,\dots}$ and $\mathcal{B}_1 \in
\mathcal{G}_0^{\sigma_1}$. On $\mathcal B_1$, by
Lemma~\ref{l:gamma}(e) with $k=1$, $\xi(\gamma_m(1))=0$ and hence also
$\xi (\gamma_m(m))=0$ for all $m< \sigma_1$. Thus, we
have \[\mathcal{B}_1 = \mathcal{B}_1 \cap \{ \xi (\gamma_m(m))=0,
\forall \, 1\le m< \sigma_1\} \in \mathcal{G}_0^{\sigma_1}\] and
\begin{align}
\indset{\mathcal{B}_1}
\widetilde{\Pr}\big(T_1=\sigma_1 | \mathcal{G}_0^{\sigma_1} \big)
= \indset{\mathcal{B}_1}
\widetilde{\Pr}\big(\xi(\gamma_{\sigma_1}(\sigma_1))=1 | \mathcal{G}_0^{\sigma_1} \big)
\geq \indset{\mathcal{B}_1} \Pr(B_0)
\end{align}
by \eqref{eq:FKG-cond}.
\smallskip

Let $\mathcal B_2=\{\gamma_{\sigma_1}(\sigma_1)\notin \mathcal C\}$.
On $\mathcal B_1 \cap \mathcal B_2^c$ we know that $T_1=\sigma_1$,
otherwise we define the stopping time
$\sigma_2=\sigma_1+\ell(\gamma_{\sigma_1} (\sigma_1))+2$. By a similar
reasoning as before, noting that ${\mathcal B}_2 \in
\mathcal{G}_0^{\sigma_2}$, we find
\begin{equation}
 \indset{\mathcal B_1 \cap \mathcal B_2}
  \widetilde{\Pr}(T_1=\sigma_2|\mathcal{G}_0^{\sigma_2})\ge
  \indset{\mathcal B_1 \cap \mathcal B_2} \Pr (B_0).
\end{equation}
By repeating the same argument, setting recursively
\begin{equation}
  \mathcal B_{k+1}=\{\gamma_{\sigma_k}(\sigma_k)\notin \mathcal C\},
  \qquad
  \sigma_{k+1} = \sigma_k + \ell(\gamma_{\sigma_k}(\sigma_k))+2
  \quad \text{on}\;\; {
    \mathcal B_1 \cap \mathcal B_2 \cap \cdots \cap \mathcal B_{k+1}},
\end{equation}
we then get
\begin{equation}
  \label{eq:sigma.k+1}
  \indset{\mathcal B_1 \cap \mathcal B_2 \cap \cdots \cap \mathcal B_{k+1}}
  \Pr\big(T_1=\sigma_{k+1}\big| \mathcal{G}_0^{\sigma_{k+1}}\big) \ge
  \indset{\mathcal B_1 \cap \mathcal B_2 \cap \cdots \cap \mathcal B_{k+1}}
  \Pr (B_0).
\end{equation}
The number of repetitions needed to find the value of $T_1$ is thus
dominated by a geometric random variable with success probability
$\Pr(B_0)>0$. Moreover, by Lemma~\ref{lem:exp-bound-perc}, the random
variables
$\sigma_{k+1}-\sigma_{k}=\ell(\gamma_{\sigma_k}(\sigma_k))+2$ have
exponential tails. By elementary considerations, this implies that
$T_1$ satisfies the desired second inequality in \eqref{eq:bounds}.
See Figure~\ref{fig:expl-branch} for an illustration of the
construction of the $T_i$. \smallskip

\begin{figure}
  \centering
  \centering
  \begin{tikzpicture}[xscale=0.65,yscale=0.65]

    \draw (-12,0) -- (-12.5,0.5)--(-12.25,0.75);
    \draw (-11.5,1) -- (-11,1.5);
    \draw (-12.5,2) -- (-13,2.5) -- (-12.75,2.75) -- (-13,3);
    \draw (-11,3.5) -- (-10.5,4) -- (-10.75,4.25) -- (-10,5);
    \draw (-11.5,4.5) -- (-11,5)--(-11.25,5.25)--(-11,5.5) ;

    \draw[very thick]  (-12,0)-- (-11.25,0.75) --
    (-12.5,2)--(-12,2.5)--(-11,3.5) --(-11.75,4.25) --
    (-11.5,4.5) -- (-12.25,5.25) -- (-12,5.5) -- (-12.5,6) -- (-11.5,7)  ;


    \draw[dashed, very thin] (-14.5,0) -- (6,0);   \draw (-15,0) node {\footnotesize $T_0$};
    \draw[dashed, very thin] (-14.5,2) -- (-9,2) -- (-5,8) -- (5.0,8);   \draw (-15,2) node {\footnotesize $T_1$};

    \draw[dashed, very thin] (-14.5,3.5) -- (-11.01,3.5);   \draw (-15,3.5) node {\footnotesize $T_2$};
    \draw[dashed, very thin] (-14.5,6) -- (-12.55,6);  \draw (-15,6) node {\footnotesize $T_3$};
    \draw[dashed, very thin] (-13.5,6.25) -- (-12.20,6.25);
    \draw[dashed, very thin] (-13.5,6.5) -- (-12.10,6.5);
    \draw[dashed, very thin] (-14.5,6.75) -- (-11.85,6.75);  \draw (-15,6.75) node {\footnotesize $T_6$};


    \draw[dashed, very thin] (1,5) -- (6,5);  \draw (6.5,5) node {\footnotesize $\sigma_1$};
    \draw (6.0,8) node {\footnotesize $\sigma_2=T_1$};

    \filldraw[fill=black]
    (1,0) circle (3pt)
    (2,1) circle (3pt)
    (0,1) circle (3pt)
    (-1,2) circle (3pt)
    (3,2) circle (3pt)
    (0,3) circle (3pt)
    (4,3) circle (3pt)
    (2,3) circle (3pt)
    (4,5) circle (3pt)
    (3,4) circle (3pt)
    (2,5) circle (3pt)
    (5,6) circle (3pt)
    (1,6) circle (3pt)
    (0,7) circle (3pt)
    (-1,8) circle (3pt)
;

    \filldraw[fill=white]
    (1,2) circle (3pt)
    (-2,3) circle (3pt)
    (-1,4) circle (3pt)
    (1,4) circle (3pt)
    (5,4) circle (3pt)
    (3,6) circle (3pt)
    (6,7) circle (3pt)
    (4,7) circle (3pt)
    (2,7) circle (3pt)
    (2,7) circle (3pt);

    \draw[thick] (0,1) -- (-1,2) -- (0,3);
    \draw[thick] (2,1) -- (4,3)-- (3,4);
    \draw[thick] (4,5) -- (5,6);
    \draw[thick] (2,5) -- (-1,8);

    \draw[-stealth,thick] (0.9,0.1) -- (0.1,0.9);
    \draw[-stealth,thick,dotted] (1.1,0.1) -- (1.9,0.9);

    \draw[-stealth,thick] (3.1,4.1) -- (3.9,4.9);
    \draw[-stealth,thick,dotted] (2.9,4.1) -- (2.1,4.9);

  \end{tikzpicture}

  \caption{``Discovering'' of the trajectory of $X$ between the
    regeneration times $T_0$ and $T_6$ in case $U =\{-1,1\}$ is shown on
    the left-hand side of the figure. On the right-hand side we zoom into
    the evolution between $T_0$ and $T_1$. On the two ``relevant sites''
    we show the realisation of $\widetilde \omega$'s using the same
    conventions as in Figure~\ref{fig:gammapaths} (in particular,
      $U(x,n)=\{(x+1,n+1),(x-1,n+1)\}$, cf.~Remark~\ref{rem:gen-neighb}).}
  \label{fig:expl-branch}
\end{figure}
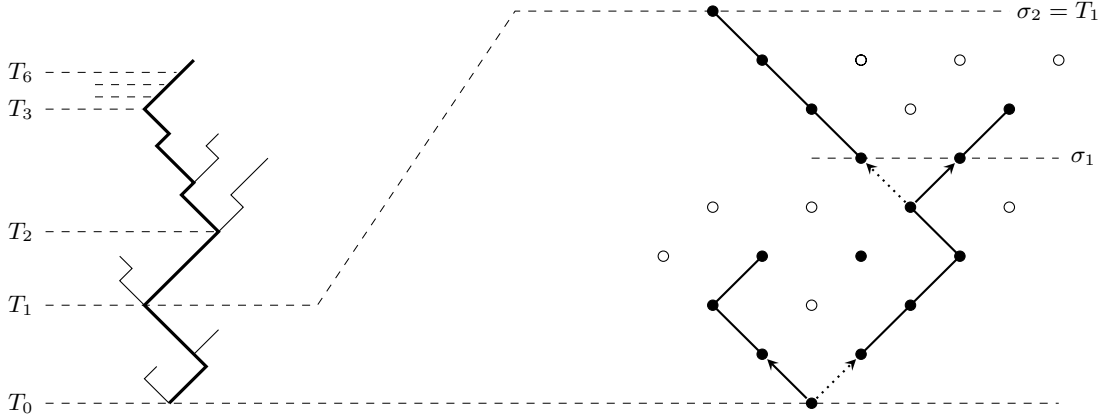

Finally we should prove that $\bigl((\tau_i,Y_i)\bigr)_{i\ge 1}$ is an
i.i.d.~sequence. Let $\theta_x$, $x\in \mathbb Z^d\times \mathbb Z$ be
the standard shift operator on $\Omega $, $(\theta_x \omega)(y)=\omega
(x+y)$. We will show that the only information we have about the
future of the environment after $T_1$ (that is about $\omega(x,n)$,
$x\in \mathbb Z^d$, $n\ge T_1$) at the instant when we discover $T_1$
is that $\xi (\gamma_{T_1}(T_1))=1$.

To formalise this let $\mathcal F_k$, $k\ge 0$, be the sigma-algebra
generated by $\mathcal G_0^{k}$ and the random variables $\xi
(\gamma_j(j))$, $0\le j \leq k$, in particular, $\mathcal F_k$
contains the information about the environments $\omega $ and
$\widetilde\omega$ that the construction discovers by the time of
checking whether $T_1=k$. Note that $T_1$ is a stopping time w.r.t.\
the filtration $(\mathcal F_k)$, write $\mathcal F_{T_1}$ for the
$T_1$-past.

We check that for any event $A\in \mathcal G_0^\infty$,
\begin{equation}
  \label{e:condition}
  \widetilde\Pr(\theta_{\gamma_{T_1}(T_1)} (A) | \mathcal F_{T_1}) =
  \widetilde\Pr(A).
\end{equation}
Indeed, if this is the case, then $(\tau_2, Y_2)$ have the same
distribution as $(\tau_1, Y_1)$ and they are independent.
Proceeding by induction then implies the i.i.d.~property.
\smallskip

To prove \eqref{e:condition} one argues similarly as before:
Pick $A' \in \mathcal F_{T_1}$, w.l.o.g.\ assume $A' \subset B_0$
(otherwise consider $A' \cap B_0$).
Fix $(z,n) \in \Z^d \times \N$. When $n \geq 2$, by construction,
\begin{align*}
  \big\{ \gamma_{T_1}(T_1) = (z,n) \big\}
  & = \bigcup_{k=1}^{\lfloor n/2 \rfloor}
  \big\{ T_1=\sigma_k=n, \gamma_n(n) = (z,n) \big\} \\
  & = \bigcup_{k=1}^{\lfloor n/2 \rfloor} \Big( \mathcal{B}_1 \cap
  \dots \cap \mathcal{B}_k \cap
  \big\{ \sigma_k=n, \gamma_n(n) = (z,n) \big\}
  \cap \{ (z,n) \to \infty \} \Big) \, ;
\end{align*}
for $n=1$, $\big\{ \gamma_{T_1}(T_1) = (z,1) \big\}
= \{ \gamma_1(1) = (z,1) \} \cap \{ (z,1) \to \infty \}$.
In particular, there exists $A'_{(z,n)} \in \mathcal G_0^n$
such that
\[
A' \cap \big\{ \gamma_{T_1}(T_1) = (z,n) \big\}
= A'_{(z,n)} \cap \{ (z,n) \to \infty \} \;\;\; \big( \subset B_0 \, \big).
\]
Thus
\begin{equation} \label{e:iidc}
  \begin{split}
& \Pr\big(\theta_{\gamma_{T_1}(T_1)} (A) \cap A' \cap B_0 \big)
= \sum_{(z,n)} \Pr\big( \theta_{(z,n)}(A) \cap A' \cap
\{ \gamma_{T_1}(T_1) = (z,n)\} \big)  \\
& = \sum_{(z,n)} \Pr\big( \theta_{(z,n)}(A) \cap \{ (z,n) \to \infty \}
\cap A'_{(z,n)} \big)  \\
& = \sum_{(z,n)} \Pr\big( \theta_{(z,n)}(A) \cap \{ (z,n) \to \infty \} \big)
\Pr\big( A'_{(z,n)} \big)  \\
& = \sum_{(z,n)} \Pr\big( \theta_{(z,n)}(A) \, | \, \{ (z,n) \to \infty \} \big)
\Pr\big( A'_{(z,n)} \big) \Pr\big( (z,n) \to \infty \big)  \\
& = \widetilde{\Pr}(A)
\sum_{(z,n)} \Pr\big( A'_{(z,n)} \cap \{ (z,n) \to \infty \} \big)
= \widetilde{\Pr}(A) \Pr(A') = \widetilde{\Pr}(A) \Pr(A' \cap B_0),
  \end{split}
\end{equation}
i.e.\ $\widetilde\Pr\big(\theta_{\gamma_{T_1}(T_1)} (A) \cap A'\big)
= \widetilde{\Pr}(A) \widetilde{\Pr}(A')$, where we used
the independence of $\mathcal G_0^n$ and $\mathcal G_n^\infty$
in the  third and in the fifth lines and
translation invariance in the fifth line. Since $A' \in \mathcal F_{T_1}$
is arbitrary, this proves \eqref{e:condition} and concludes the proof
of Lemma~\ref{lem:regenstructure}. \hfill \qed

\begin{remark}[Continuity of $\sigma(p)$, cf.~Remark~\ref{rem:sigma.p}]
  \label{rem:sigma-cont}
  The construction in the proof of Lemma~\ref{lem:regenstructure}
  also shows that the functions
  \begin{equation}
    \label{eq:ptoEp-cont}
    p \mapsto \E_p\big[ Y_{1,1}^2 \big], \quad
    p \mapsto \E_p[\tau_1] \quad
    \text{and in particular} \;\; p \mapsto \sigma^2(p)
  \end{equation}
  are continuous on $(p_c,1]$.
\end{remark}
\begin{proof}
  For fixed $z \in \Z, n \in \N$ note that by construction
 $\{ Y_{1,1}=z, \tau_1=n \} = D_{z,n} \cap \{ (z,n) \to \infty \}$
  where $D_{z,n} \in \sigma\big( \omega(x,m), \widetilde\omega(x,m),
  (x,m) \in B_n \big) \subset \mathcal{G}_0^n$
  (with $B_n := \{ (x,m) : \norm{x} \leq n, 0 \leq m<n \}$) can be expressed
  as a finite union
  \begin{align*}
  D_{z,n} = \bigcup_{\overline{\omega} \in C(z,n)}
  \Big(
  \big\{ & \omega(x,m) = \overline{\omega}(x,m) \; \text{for} \;
  \norm{x} \leq n, m < n \big\} \\[-3ex]
  & \cap \big\{ (\widetilde\omega(x,m) : \norm{x} \leq n, m < n)
  \in \widetilde{C}(z,n,\overline{\omega}) \big\}
  \Big)
  \end{align*}
  for certain $C(z,n) \subset \{0,1\}^{B_n}$
  and $\widetilde{C}(z,n,\overline{\omega}) \subset
  \{ \text{permutations of $1,\dots,n$} \}^{B_n}$.
  Thus,
  \[
  p \mapsto \Pr_p\big(Y_{1,1}=z, \tau_1=n\big)
  = \Pr_p(D_{z,n}) \Pr_p\big( (z,n) \to \infty\big)
  = \Pr_p(D_{z,n}) \Pr_p\big( (\boldsymbol 0,0) \to \infty\big)
  \]
  is continuous on $(p_c,1]$ for any $(z,n) \in \Z \times \N$ :
  Since $D_{z,n}$ depends only on
  finitely many coordinates of $\omega$ and $\widetilde\omega$,
  continuity of $p \mapsto \Pr_p(D_{z,n})$ is obvious e.g.\
  from a simple coupling argument; continuity of
  $p \mapsto \Pr_p\big( (\boldsymbol 0,0) \to \infty\big)$ is guaranteed by
  \citep[][Thm.~2]{GrimmetHiemer:02}.
  Combining this with exponential tail bounds for $|Y_{1,1}| \leq \tau_1$
  that can be chosen uniform in $p \in [p_c+\delta,1]$ for any $\delta>0$
  (cf.~Lemma~\ref{lem:exp-bound-perc} and
  Remark~\ref{rem:exp-bound-perc-unif} below) proves \eqref{eq:ptoEp-cont}.
\end{proof}

\section[Two walks on the same cluster and quenched CLT]
  {Joint dynamics of two  walks on the same realisation
    of the cluster and the quenched CLT}
\label{sec:joint-dynamics}

In this section we study the joint dynamics of two walks on the same
realisation of the cluster $\mathcal C$, in order to show the quenched
CLT, Theorem~\ref{thm:quenchedCLT}.

For $x,x'\in \mathbb Z^d$, let $B_{x,x'}$ be the event
$\{\xi_0(x)=\xi_0(x')=1\}$. Conditioned on $\omega $ and $B_{x,x'}$, let
$X\coloneqq (X_n)_n$ and $X' \coloneqq (X'_n)_n$ be two \emph{independent}
random walks with
transition probabilities \eqref{eq:defXdynamics} started on $(x,0)$ and
$(x',0)$ respectively.
Observe that $X$ and $X'$ take their steps independently but in the same
environment.
Note also that, unlike true ancestral lines, the random walks $X$, $X'$  can
meet and then separate again.

We now extend the local construction of the random walk from
Subsection~\ref{ss:localconst} to the two-walk case. Assume that in
addition, a collection of independent random permutations
$\widetilde\omega'=(\widetilde\omega'{(x,n)})_{(x,n) \in \mathbb Z^d
  \times \mathbb Z}$ with the same distribution as $\widetilde\omega$
on the same probability space $(\Omega ,\mathcal A, \mathbb P)$ is
given, and define paths $\gamma_k'^{(x,n)}$ analogously as
$\gamma_k^{(x,n)}$ using $\widetilde\omega'$ instead of
$\widetilde\omega$ but the same $\omega $. Note that for given $n$ and
$k$, the construction of $\gamma_k^{(x,n)}$ and $\gamma_k'^{(x,n)}$ is
measurable with respect to
\begin{align}
  \label{eq:defGnnk2}
  \widehat{\mathcal{G}}_{n}^{n+k}\coloneqq
  \sigma\big(\omega(y,i),\widetilde\omega{(y,i)},\widetilde\omega'{(y,i)} : y
  \in \Z^d, n \leq i < n+k \big).
\end{align}

On $B_{x,x'}$, using the same reasoning as in the previous section, we may
couple the random walks $X$, $X'$ started from $x$, $x'$ with $\omega $,
$\widetilde\omega$ and $\widetilde\omega'$ by
\begin{equation}
  (X_k,k)=\lim_{n\to \infty} \gamma^{(x,0)}_n(k),\qquad
  (X_k',k)=\lim_{n\to \infty} \gamma_n'^{(x',0)}(k).
\end{equation}

\subsection{Joint regeneration structure of two random walks}
\label{sec:joint-renew}

The individual regeneration sequences are defined as in Section~\ref{sec:reg-struct}.
We now define a joint regeneration sequence for the pair $X$, $X'$.
We set $T_0\coloneqq0$, $T'_0\coloneqq0$ and for $r \in \Z_+$ put
\begin{align}
  \label{eq:defRi}
  T_{r+1} & \coloneqq \inf \{ n > T_{r} : \xi(\gamma_n^{(x,0)}(n) )=1 \}, \\
  \label{eq:defRi'}
  T'_{r+1} & \coloneqq \inf \{ n > T'_{r} : \xi(\gamma_n'^{(x',0)}(n) )=1 \}.
\end{align}
Note that if $(X_0,X'_0)=(x,x)$, then under $\Pr(\cdot | B_{x,x})$,
$(X_{T_{r+1}}-X_{T_r},T_{r+1}-T_r)_r$ is an i.i.d.\ sequence and
$(X'_{T'_{r+1}}-X'_{T'_r},T'_{r+1}-T'_r)_r$ has the same law but the two objects
are of course not independent because both build on the same cluster given by
the same $\xi$.

Now we define the sequence of simultaneous regeneration times. We set
$J_0\coloneqq0$, $J'_0\coloneqq0$ and for $m \in \Z_+$ let
\begin{align}
  \label{eq:Jmdef}
  J_{m+1} &  \coloneqq \min \big\{ j > J_{m} : T_j = T'_{j'} \; \text{for some}\; j'>J'_{m} \big\}, \\
  \label{eq:Jm'def}
  J'_{m+1} & \coloneqq \min \big\{ j' > J'_{m} : T'_{j'} = T_j \; \text{for some}\; j > J_{m} \big\},
\end{align}
then
\begin{align} \label{eq:12}
  T^{\Sim}_m \coloneqq T_{J_m} = T'_{J'_m}, \quad m=0,1,2,\dots
\end{align}
is the sequence of simultaneous regeneration times. Note that
\begin{align}
  T^{\Sim}_m = \min \Big( \{ T_j : T_j > T^{\Sim}_{m-1} \}
  \cap \{ T'_j : T'_j > T^{\Sim}_{m-1} \} \Big).
\end{align}

As in the one walk case we write for the increments $Y_k\coloneqq X_{T_k}-X_{T_{k-1}}$,
$\tau_k\coloneqq T_k-T_{k-1}$, $Y'_k\coloneqq X'_{T'_k}-X'_{T'_{k-1}}$,
$\tau'_k\coloneqq T'_k-T'_{k-1}$
and furthermore  we define
\begin{align}
  \label{eq:deftildeXm}
  & \widetilde{X}_m \coloneqq X_{T_m}, \quad \widetilde{X}'_m \coloneqq X'_{T'_m}, \quad m\in\Z_+, \\
  \label{eq:defhatXell}
  & \widehat{X}_\ell \coloneqq X_{T^{\Sim}_\ell} = \widetilde{X}_{J_\ell}
  = X_{T_{J_\ell}}, \quad
  \widehat{X}'_\ell \coloneqq X'_{T^{\Sim}_\ell} = \widetilde X'_{J'_\ell} =
  X'_{T'_{J'_\ell}}, \quad \ell \in \Z_+.
\end{align}
Note that $\widetilde{X}_m$, $\widetilde{X}'_m$ will typically refer to the two
walks $X$ and $X'$ at different real-time instants.

Let us denote the ``pieces between simultaneous regenerations'' by
\begin{align}
  \label{eq:Xi_m}
  \Xi_m \coloneqq \Big( \big( Y_k, \tau_k \big)_{k=J_{m-1}+1}^{J_m},
  \big( Y'_k, \tau'_k \big)_{k=J'_{m-1}+1}^{J'_m},
  \widehat X_m, \widehat X'_m \Big),
  \quad m=1,2,\dots.
\end{align}
Note that $\Xi_m$ takes values in $\mathbbm F \times \mathbbm F \times \Z^d
\times \Z^d$, where $\mathbbm F \coloneqq \cup_{n=1}^\infty (\Z^d \times \N)^n$.
\smallskip

The following result is the ``joint'' version of bound
\eqref{eq:bounds} in Lemma~\ref{lem:regenstructure}. Heuristically,
since the individual regeneration times have exponential tails and
immediate joint regeneration has ``positive'' probability one can use
a ``restart''-argument to construct joint regenerations. Because of
the dependence of the two walks the proof that we actually give is
somewhat more complicated than these heuristics.

\begin{lemma}[Exponential tail bounds for joint regeneration times]
  \label{lemma:jointrenewaltails}
  There exist constants $C,c \in(0, \infty)$ such that
  \begin{align}
    \Pr\big( T^{\Sim}_1 \geq k
    \,\big|\, X_0=x, X'_0=x', B_{x,x'} \big) \leq C e^{-ck},
    \quad\forall k\in \N, x,x'\in\Z^d.
  \end{align}
\end{lemma}
\begin{proof}
  Let $\gamma_k =\gamma^{(x,0)}_k$ and $\gamma '_k=\gamma_k'^{(x',0)}$.
  The proof is a variant of the proof of Lemma~\ref{lem:regenstructure},
  but one should be a little bit more careful not to ``discover too many
  sites where $\xi$ is zero''. More precisely one must not check whether
  $\xi (\gamma_k(k))=1$ and $\xi (\gamma '_k(k))=1$ for all $k$. We
  proceed as follows. We first check whether $\xi (\gamma_1(1))=1$.
  If this is not the case, then we \emph{do not check} $\xi(\gamma '_1(1))$, set
  $\sigma_1=\ell(\gamma_1(1))+3$.  When $\xi (\gamma_1(1))=1$, we check
  $\xi (\gamma'_1(1))$. When it is $1$, then we are done and
  $T^{\Sim}_1=1$. When it is $0$, we set
  $\sigma_1=\ell(\gamma '_1(1))+3$.

  If we are not done, we proceed with the local construction of
  $\gamma_k $ and $\gamma_k'$, but
  do not check  any other value of $\xi$ until reaching time $\sigma_1$ (as
    it is useless, we need first ``discover locally'' the fact that one of the
    $\xi$'s we checked before was zero). At time $\sigma_1$ we have this
  information, so we check $\xi (\gamma_{\sigma_1}(\sigma_1))$ first.
  If it is zero, we set
  $\sigma_2 = \sigma_1 +\ell (\gamma_{\sigma_1}(\sigma_1))+2$. If it is one, we
  check also $\xi (\gamma'_{\sigma_1}(\sigma_1))$. When also this
  value is one, we are done, $T^{\Sim}_1=\sigma_1$. Otherwise we
  set $\sigma_2 = \sigma_1+\ell (\gamma'_{\sigma_1}(\sigma_1))+2$.
  If we are not done, we continue the local construction up to the time
  $\sigma_2$ without checking any $\xi$'s. At time
  $\sigma_2$ we check two end points of $\gamma$'s as before,
  eventually defining $\sigma_3$, etc.

  Let, similarly as before, $\widehat{\mathcal F}_k$ be the $\sigma
  $-algebra generated by $\widehat{\mathcal G}_0^k$ and all the
  additional information about $\xi $'s discovered by this algorithm
  (strictly) before time $k$. To estimate how many steps of the
  algorithm are necessary, we claim that when $\sigma_k$ is defined
  \begin{equation}
    \mathbb P\left(\xi (\gamma_{\sigma_k}(\sigma_k))=1= \xi
      (\gamma'_{\sigma_k}(\sigma_k))|
      \widehat{\mathcal F}_{\sigma_k},B_{x,x'}\right) \ge \Pr(B_0)^2.
  \end{equation}
  Indeed, this follows again by the FKG inequality, one should only
  observe that any negative information contained in the conditioning (that
    is the knowledge $\xi (z)=0$ for some $z$) is contained  in
  $\mathcal G_0^{\sigma_k}$ and can thus be removed from the conditioning,
  similarly to \eqref{eq:sigma.k+1}.

  The number of steps of the algorithm is thus dominated by a geometric
  random variable. Moreover, as the random variables
  $\sigma_i-\sigma_{i-1}$ have exponential tails, the claim of the lemma
  follows as before.
\end{proof}

We next show that $(\Xi_m)_{m\in\Z_+}$ defined in \eqref{eq:Xi_m} form a
(discrete) Markov chain.
\begin{lemma}
  Let $x, x' \in \Z^d$, $X_0=x$, $X'_0=x'$, put
  $\Xi_0\coloneqq(\alpha,\alpha',x,x')$ with arbitrary
  $\alpha, \alpha' \in \mathbb F$.
  Then under
  $\Pr(\cdot | B_{x,x'})$, $(\Xi_m)_{m\in\Z_+}$ is a (discrete)
  Markov chain, its transition probability function
  \begin{align}
    \label{e:independencealpha}
    \Psi^{\mathrm{joint}}\big( (\alpha,\alpha',x,x'),
    (\beta,\beta',y,y') \big)
    =: \Psi^{\mathrm{joint}}\big( (x,x'), (\beta,\beta',y,y') \big)
  \end{align}
  depends in its first argument only on the last two coordinates $(x,x')$,
  not on $(\alpha,\alpha')$, and has a spatial homogeneity property:
  \begin{align}
\Psi^{\mathrm{joint}}\big( (x,x'), (\beta,\beta',y,y') \big) = \Psi^{\mathrm{joint}}\big( (x+z,x'+z), (\beta,\beta',y+z,y'+z) \big).
  \end{align}
\end{lemma}
\begin{proof}
  The proof is a straightforward adaptation of arguments given around
  \eqref{e:condition}--\eqref{e:iidc} in the proof of
  Lemma~\ref{lem:regenstructure}. Therefore we do not repeat it here.
\end{proof}

\begin{remark}
  \label{rem:Xichain}
  1.\ In particular,
  $(\widehat{X}_\ell, \widehat{X}'_\ell)_\ell$ is in itself a Markov chain
  on $\Z^d \times \Z^d$ with transition probability
  \begin{equation}
    \label{eq:PSihatjoint}
    \widehat{\Psi}^{\mathrm{joint}}\left( (x,x'), (y,y') \big)
  \coloneqq \Psi^{\mathrm{joint}}\big( (x,x'), \mathbbm F\times \mathbbm F \times
  \{(y,y')\}\right).
  \end{equation}

  2.\ The full path $(\Xi_m)_m$ contains the same amount of information
  as the pair of sequences $\big((Y_i,\tau_i)_i, (Y'_i,\tau'_i)_i\big)$ since
  these can be reconstructed from a full $\Xi$-path.
\end{remark}

\smallskip

We want to compare the joint distribution of $(X,X')$ run in the same
environment with the distribution of two walks run in two independent
copies of the environment. To this end we consider the same
construction as above, except that now we let $X'$ use $\xi'$ built on
$\omega'$, an independent copy of $\omega $. That is, the sequences
$(X_{T_{i+1}}-X_{T_i},T_{i+1}-T_i)_i$ and
$(X'_{T'_{i+1}}-X'_{T'_i},T'_{i+1}-T'_i)_i$ are now independent copies
since they use independent realisations of the medium. Obviously, then
$(\Xi_m)_{m\in\Z_+}$ as defined in \eqref{eq:Xi_m} is a Markov chain
and we denote its transition probability function by
\begin{align}
  \Psi^{\indep}\big( (\alpha,\alpha',x,x'), (\beta,\beta',y,y')
  \big)
  \eqqcolon \Psi^{\indep}\big( (x,x'), (\beta,\beta',y,y') \big).
\end{align}
It again only depends on the last two coordinates $(x,x')$, not on
$(\alpha,\alpha')$, and has the same spatial homogeneity as
$\Psi^{\mathrm{joint}}$. Under $\Psi^{\indep}$, the sequence
$(T^{\Sim}_\ell-T^{\Sim}_{\ell-1})_\ell$ is i.i.d., the law of
$T^{\Sim}_1$ does not depend on the spatial separation (it is simply
the law of the first joint renewal of two independent copies of a
renewal process whose waiting time distribution is aperiodic and has
exponential tails). Finally, similarly as in \eqref{eq:PSihatjoint} we
define
\begin{equation}
\label{eq:PSihatindep}
  \widehat{\Psi}^{\indep}\big((x,x'), (y,y') \big) \coloneqq
  \Psi^{\indep}\big((x,x'), \mathbbm F\times \mathbbm F\times\{(y,y')\} \big).
\end{equation}

The next lemma allows us to compare $\Psi^\joint$ and $\Psi^\indep$.
\begin{lemma}[Total variation distance of $\Psi^\joint$ and $\Psi^\indep$] There
  exist constants $c, C >0$ such that
  \label{lemma:couplingjointregen}
  \begin{align}
    \norm{\Psi^\joint \big((x,x'), \cdot \big)-
      \Psi^{\indep}\big((x,x'), \cdot \big)}_{\textnormal{TV}}
    \le Ce^{-c \norm{x-x'}}, \text{ for all } x,x' \in \Z^d,
  \end{align}
where $\norm{\cdot}_{\textnormal{TV}}$ denotes the total variation norm.
\end{lemma}
\begin{proof}
  The proof is an adaptation of the proof of Lemma~3.2(i) in
  \citep{Valesin:2010}. We give it in the case $x \in \Z^d$ of the form
  $x=(x_1,0,\dots,0)$ for some positive $x_1$ and $x'=\mathbf 0$. The general
  case has more complicated notation but the same arguments.

  Let
  \begin{align}
    \label{eq:6}
    \Omega_i=\left\{\left(\omega_i(z,n),\widetilde\omega_i(z,n)\right) :
      (z,n)\in \Z^d \times \Z_+\right\},
    \quad i=1,2
  \end{align}
  be two independent families of independent collections of random
  variables, where the random variables $\omega_{i}(z,n)$ are i.i.d.\
  Bernoulli distributed with parameter $p>p_c$ (where $p_c$ is the
  critical parameter for oriented percolation) and the random
  variables $\widetilde\omega_i(z,n)$ are i.i.d.~random permutations
  of the sets $U(z,n)$ (defined in \eqref{Udef}).

  Furthermore define
  $\Omega_3 = \left\{\left(\omega_3(z,n),\widetilde \omega_3(z,n)\right):
    (z,n)\in \Z^d\times \Z\right\}$ by setting for $z =(z_1,\dots,z_d)$
  \begin{align*}
    \left(\omega_3(z,n),\widetilde \omega_3(z,n)\right) \coloneqq
    \begin{cases}
      (\omega_1(z,n),\widetilde \omega_1(z,n)) & :  z_1 \le x_1/2, \\
      (\omega_2(z,n),\widetilde \omega_2(z,n))  & :  z_1 > x_1/2.
    \end{cases}
  \end{align*}
  Then of course $\Omega_1, \Omega_2$ and $\Omega_3$ have the same
  distribution and on each of the families we can define the random walks
  by the local construction of Section~\ref{sec:reg-struct}. To distinguish
  these walks throughout the proof we will need to denote several
  variables that we introduced earlier as functions of the $\Omega_i$'s.
  In particular we write
  \begin{align*}
    \xi_n(x;\Omega_i) & \text{ for $\xi_n(x)$ constructed using
      $\Omega_i$},\\
    \ell(x,n;\Omega_i) & \text{ for $\ell(x,n)$ constructed using $\Omega_i$},
    \\
    \gamma^{(x,n)}_k
    (m;\Omega_i) & \text{ for the path $\gamma^{(x,n)}_k(m)$ obtained
      using $\Omega_i$}.
  \end{align*}
  Furthermore for $i,j \in \{1,2,3\}$ (we will consider the two cases $i=j=3$ or
    $i=1$ and $j=2$) we set
  \begin{align*}
    &B_{x,x'} (\Omega_i,\Omega_j)  \coloneqq \{\xi_0(x;\Omega_i) = 1 =
      \xi_0(x';\Omega_j)\},
    \\
    &T^{\Sim}_{i,j}  \coloneqq T^{\Sim}(\Omega_i,\Omega_j)   : =
    \inf\bigl\{n \ge 1: \xi_n\bigl(\gamma^{(x,n)}_n (n;\Omega_i);\Omega_i\bigr) =
      \xi_n\bigl(\gamma^{(x',n)}_n(n;\Omega_j);\Omega_j\bigr) =1  \bigr\}.
  \end{align*}
  Note that on $B_{x,x'} (\Omega_3,\Omega_3)$ the regeneration time
  $T^{\Sim}_{3,3}$ is the simultaneous regeneration time $T_1^{\Sim}$
  defined as in \eqref{eq:12} using $\Omega_3$ and $T_{1,2}^{\Sim}$ is
  the first simultaneous regeneration time of two independent walks
  defined on $\Omega_1$ and $\Omega_2$. In keeping with
  (\ref{eq:Jmdef}--\ref{eq:Jm'def}) we will write $J_1(\Omega_i,
  \Omega_j)$ and $J_1'(\Omega_i, \Omega_j)$ for the number of
  individual renewals until the first joint renewal of the first,
  respectively, the second walk when the first walk uses $\Omega_i$
  and the second $\Omega_j$.

  Note also that there
  are constants $c,C \in (0, \infty)$ such that, $\forall i,j \in \{1,2,3\}$,
  \begin{align} \label{eq:13}
    \Pr\left(T^{\Sim}_{i,j} > r | B_{x,x'}(\Omega_i,\Omega_j) \right)
    \le C e^{-c r}.
  \end{align}
  For $i=j=3$ this assertion was shown in Lemma~\ref{lemma:jointrenewaltails}.
  For $i=1$ and $j=2$ the inequality is true since the individual regeneration
  times of two independent random walks are aperiodic and by
  Lemma~\ref{lem:regenstructure} have exponentially decaying tails.

  Recall the definition of $\Xi_m$ in \eqref{eq:Xi_m} and define for $i,j \in
  \{1,2,3\}$
  \begin{align*}
    \Xi_1(\Omega_i,\Omega_j)\coloneqq \big(
    ( Y_k(\Omega_i), \tau_k(\Omega_i) )_{k=1}^{J_1(\Omega_i, \Omega_j)},
    (Y_k'(\Omega_j), \tau_k'(\Omega_j) )_{k=1}^{J_1'(\Omega_i, \Omega_j)},
    X_{T^{\Sim}_{i,j}}(\Omega_i),
    X'_{T^{\Sim}_{i,j}}(\Omega_j) \big).
  \end{align*}
  Furthermore define, with some cemetery state $\Delta$,
  \begin{align*}
    \Xi^\joint_{x,x'}\coloneqq
  \begin{cases}
    \Xi_1(\Omega_3,\Omega_3), &  \text{ if }
    \xi_0(x;\Omega_3)=\xi_0(x';\Omega_3)=1,
    \\
    \Delta, &  \text{ otherwise, }
  \end{cases} \\
  \intertext{and} \Xi^{\indep}_{x,x'} \coloneqq
  \begin{cases}
    \Xi_1(\Omega_1,\Omega_2), &  \text{ if }
    \xi_0(x;\Omega_1)=\xi_0(x';\Omega_2)=1,
    \\
    \Delta, &  \text{ otherwise. }
  \end{cases}
\end{align*}
Recall that we are considering the supercritical case $p>p_c$, hence
the percolation probability $p_\infty\coloneqq\Pr(\xi_0(\boldsymbol
0)=1)$ is strictly positive. Because of positive correlations, we have
\begin{align}\label{eq:14}
  \Pr\left(\Xi^\joint_{x,x'} \ne \Delta\right), \Pr\left(\Xi^\indep_{x,x'} \ne
    \Delta\right) \ge p_\infty^2
\end{align}
uniformly in $x,x'$. Furthermore we have
\begin{align*}
  \Psi^\joint\bigl((x,x'),\cdot\bigr) & = \Pr\left(\Xi^\joint_{x,x'} = \cdot\, | \,
  \Xi^\joint_{x,x'} \ne \Delta\right)\\ \intertext{and}
  \Psi^\indep\bigl((x,x'),\cdot\bigr) & = \Pr\left(\Xi^\indep_{x,x'} = \cdot\, | \,
  \Xi^\indep_{x,x'} \ne \Delta\right).
\end{align*}

Define $n^*=[x_1/2]$ and set
\begin{align*}
  L_1 & = \{\ell(x,0;\Omega_1) \vee \ell(x',0;\Omega_2) \le n^* \} \\
  L_2 & = \{\xi_0(x;\Omega_1)=\xi_0(x';\Omega_2)=\xi_0(x;\Omega_3) =
  \xi_0(x';\Omega_3)=1, T_{3,3}^{\Sim} \le n^*,
  T_{1,2}^{\Sim} \le n^* \}.
\end{align*}
By definition of $\Omega_3$ and $n^*$ we have
\begin{align} \label{eq:4}
  \gamma^{(x',0)}_n(n;\Omega_1)=\gamma^{(x',0)}_n (n;\Omega_3)\; \text{and} \;
  \gamma^{(x,0)}_n(n;\Omega_2)=\gamma^{(x,0)}_n (n;\Omega_3)\; \text{ for
    all } n \in
  \{0,\dots, n^*\}.
\end{align}
Furthermore on $L_1 \cup L_2$ we have $\Xi^\joint_{x,x'}=\Xi^\indep_{x,x'}$. To
see this note that on $L_1$ we have
$\Xi^\joint_{x,x'}=\Xi^\indep_{x,x'}=\Delta$. On $L_2$ we have
$T_{1,2}^{\Sim}= T_{3,3}^{\Sim}$ and since this is smaller than
$n^*$ we obtain by \eqref{eq:4} that $\Xi_{x,x'}^\joint=\Xi_{x,x'}^\indep$.

The complement of $L_1 \cup L_2$ is contained in the union of the events
\begin{align*}
  \{n^* < \ell(x,0;\Omega_1) <\infty\},   \{n^* < \ell(x,0;\Omega_3)
  <\infty\}, \\
  \{n^* < \ell(x',0;\Omega_2) <\infty\},   \{n^* < \ell(x',0;\Omega_3)
  <\infty\},  \\
  \{\xi_0(x;\Omega_1)=\xi_0(x';\Omega_2)=1, T_{1,2}^{\Sim}>n^*\}, \\
  \{\xi_0(x;\Omega_3)=\xi_0(x';\Omega_3)=1, T_{3,3}^{\Sim}>n^*\},
\end{align*}
each of which is exponentially decreasing in $\norm{x-x'}=x_1$. For
the events in the first two lines this follows by
Lemma~\ref{lem:exp-bound-perc}, whereas for the events in the last two
lines this is a consequence of \eqref{eq:13}. Thus, there are $c, C
\in (0,\infty)$ such that
\begin{align*}
  \sum_{w \in \mathcal W \cup \{\Delta\} } \Abs{\Pr\left(\Xi_{x,x'}^\joint= w \right)-
    \Pr\left(\Xi_{x,x'}^\indep= w \right)} =
  \Pr\left(\Xi_{x,x'}^\joint \ne \Xi_{x,x'}^\indep\right) \le C e^{-c
    x_1},
\end{align*}
where $\mathcal W\coloneqq \mathbbm F \times \mathbbm  F \times\Z^d \times
\Z^d$.
Now, as in \citet{Valesin:2010} in the display after (3.9) on p.~2236, it
follows that
\begin{align*}
  \norm{\Psi^\joint \big((x,x'), \cdot \big) & - \Psi^{\indep}\big((x,x'), \cdot
    \big)}_{\textnormal{TV}} \\ & = \frac12 \sum_{w \in \mathcal W } \Abs{
    \frac{\Pr\left(\Xi^\joint_{x,x'} = w\right)}{\Pr\left(\Xi^\joint_{x,x'} \ne
        \Delta\right)} - \frac{\Pr\left(\Xi^\indep_{x,x'} =
        w\right)}{\Pr\left(\Xi^\indep_{x,x'} \ne
        \Delta\right)} } \\
  & \le \frac1{2\Pr\left(\Xi^\joint_{x,x'} \ne \Delta\right)} \sum_{w \in
    \mathcal W } \Abs{\Pr\left(\Xi^\joint_{x,x'} = w\right) -
    \Pr\left(\Xi^\indep_{x,x'} = w\right)} \\
  & \qquad \qquad + \frac12 \Abs{\frac1{\Pr\left(\Xi^\joint_{x,x'} \ne
        \Delta\right)}- \frac1{\Pr\left(\Xi^\indep_{x,x'} \ne \Delta\right)}}
  \sum_{w \in \mathcal W } \Pr\left(\Xi^\indep_{x,x'} =
    w\right) \\
  & \le \frac{Ce^{-c \norm{x-x'}}}{\Pr\left(\Xi^\joint_{x,x'} \ne \Delta\right)}
  + \frac12 \frac{\Abs{\Pr\left(\Xi^\indep_{x,x'} \ne \Delta\right) -
      \Pr\left(\Xi^\joint_{x,x'} \ne \Delta\right)}}{\Pr\left(\Xi^\joint_{x,x'}
      \ne \Delta\right) \Pr\left(\Xi^\indep_{x,x'} \ne \Delta\right)}
  \\
    & \le \frac{Ce^{-c \norm{x-x'}}}{\Pr\left(\Xi^\joint_{x,x'} \ne \Delta\right)}
  + \frac{Ce^{-c \norm{x-x'}}}{\Pr\left(\Xi^\joint_{x,x'}
      \ne \Delta\right) \Pr\left(\Xi^\indep_{x,x'} \ne \Delta\right)}.
\end{align*}
By \eqref{eq:14} both denominators in the last line of the above
display are bounded away from zero uniformly in $x,x'$. Thus, for
suitably chosen $C,c$ the assertion of the lemma follows in the case
$x=(x_1,0,\dots,0)$.
\end{proof}

\smallskip

\subsection{Coupling of $\Psi^\joint$ and $\Psi^\indep$}
\label{sec:coupling}

The following lemma is our ``target lemma'', which forms the core of
the proof of Theorem~\ref{thm:quenchedCLT}.
\begin{lemma}
  \label{lemma:qCLTars}
  There exists $b>0$ and a non-trivial centred $d$-dimensional
  normal law $\widetilde{\Phi}$ such that for $f : \R^d\to\R$
  bounded and Lipschitz we have
  \begin{align}
    \label{eq:L2regenCLT}
    \E\left[ \Big( E_\omega\bigl[f\big(\widetilde{X}_m/\sqrt{m}\big)\bigr]
      -\widetilde{\Phi}(f) \Big)^2 \right] \leq C_f \, m^{-b}.
  \end{align}
\end{lemma}

We will show this lemma by coupling two Markov chains with
transition matrices $\Psi^\joint$ and $\Psi^\indep$, using
Lemma~\ref{lemma:couplingjointregen}. We need a few technical lemmata beforehand.
The first one gives standard estimates on exit distribution from an annulus.

\begin{lemma}
  \label{lemma:harmonicballsd3}
  Write for $r>0$
  \begin{equation}
    \label{eq:hittingtimedistouter}
    \begin{split}
      h(r) &\coloneqq \inf\{ k \in \Z_+ : \norm{\widehat X_k -\widehat X'_k}_2 \leq r\},\\
      H(r) &\coloneqq \inf\{ k \in \Z_+ : \norm{\widehat X_k -\widehat X'_k}_2 \geq r\},
    \end{split}
  \end{equation}
  where $\norm{\cdot}_2$ denotes the Euclidean norm on $\Z^d$,
  and set for $r_1 < r < r_2$
  \begin{equation}
    f_d(r;r_1,r_2)=
    \begin{cases}
      \frac{r_1^{2-d}-r^{2-d}}{r_1^{2-d}-r_2^{2-d}},
      &\text{when $d\ge 3$},\\[1.2ex]
      \frac{\log r - \log r_1}{\log r_2 - \log r_1},
      &\text{when $d=2$}, \\[1.2ex]
      \frac{r-r_1}{r_2-r_1},
      &\text{when $d=1$}.
    \end{cases}
  \end{equation}
  For every $\varepsilon >0$ there are (large) $R$ and $\widetilde R$
  such that for all $ r_2 > r_1 > R$ with $r_2-r_1 > \widetilde R$
  and $x,y\in \mathbb Z^d$ satisfying
  $r_1<r=\norm{x-y}_2<r_2$
  \begin{equation}
    \label{eq:harmonicballsd3}
    (1-\varepsilon)f_d(r;r_1,r_2)\le
    \Pr^\indep_{x,y}\big(H(r_2)  < h(r_1) \big) \le
    (1+\varepsilon)f_d(r;r_1,r_2).
  \end{equation}
\end{lemma}
\begin{remark}
  In the rest of the paper, we use for simplicity the $\sup$-norm
  $\norm{x}=\max_{1 \le i \le d} |x_i|$ instead of the Euclidean norm
  $\norm{x}_2$. Since all norms on $\R^d$ are equivalent, we
  can and shall translate between them by an appropriate adjustments
  of constants. We will assume this implicitly below when applying
  Lemma~\ref{lemma:harmonicballsd3}.
\end{remark}
\begin{proof}[Proof of Lemma~\ref{lemma:harmonicballsd3}]
  From the construction of the transition probability $\Psi^{\indep}$
  it follows that under $\Pr^\indep$ the Markov chain $(\widehat
  X_n-\widehat X'_n)_n$ is a random walk on $\mathbb Z^d$ with
  i.i.d.~increments whose distribution is symmetric (and thus centred)
  and has a finite variance. The claim is then a direct consequence of
  the usual invariance principle and exit probabilities from an
  annulus by a $d$-dimensional Brownian motion \citep[see e.g.\
  ][Thm.~3.18]{MoertersPeres2010}.
 \end{proof}

We  use $\Pr^\joint_{x,y}$ to denote the distribution of
the canonical Markov chain $\Xi$ with transition probabilities
$\Psi^\joint$ started from $\Xi_0 = (\alpha ,\alpha ',x,y)$. Note that by
\eqref{e:independencealpha} this distribution does not depend on
$\alpha, \alpha '$. Similarly, $\Pr^\indep_{x,y}$ denotes the law of the
chain with the transition matrix $\Psi^\indep$. In both cases, with a
slight abuse of notation,
$\widetilde X, \widetilde X', \widehat X, \widehat X'$ denote the
corresponding underlying chains which can be read from $\Xi$, see
Remark~\ref{rem:Xichain}. In particular, under $\Pr^\indep$,
$(\widehat X_n,\widehat X'_n)_n$ is a Markov chain with transition
probability $\widehat \Psi^\indep$, given in~\eqref{eq:PSihatindep}.

From now on, we assume in this section that $d\ge 2$ and complete the
proof of Lemma~\ref{lemma:qCLTars} under this assumption. The case $d=1$
requires different arguments that are postponed to
Section~\ref{sec:cased1}.

As we do not have a good control on the difference between
$\Psi^\indep((x,x'),\cdot)$ and $\Psi^\joint((x,x'),\cdot)$ when $\|x-x'\|$ is
small, we need to ensure that $\widehat X$ and $\widehat X'$ separate
sufficiently quickly under $\Pr^\joint$. This is shown in the next lemma:

\begin{lemma}[Separation lemma]
  \label{lemma:separation1}
  Let $d\ge 2$ and let $H(r)$ be as in \eqref{eq:hittingtimedistouter}. There
  are $b_1, b_2 \in (0,1/2)$, $b_3 > 0$, $b_4 \in (0,1)$ such that for $n$ large
  enough,
  \begin{equation}
    \Pr^\joint_{\mathbf 0, \mathbf 0}\big( H(n^{b_1}) \geq n^{b_2} \big) \leq
    e^{-b_3 n^{b_4}}.
  \end{equation}
\end{lemma}
\begin{proof} We split the proof in six steps.

\medskip
\noindent
  \emph{Step~1.\ }
  We first observe that there exists a (small) $\varepsilon_1>0$
  and $b_4 \in (0,1/2)$, $b_5>0$
  such that
  \begin{equation}
    \label{eq:lemma:separation1:1}
    \Pr^\joint_{x,y}\big(H(\varepsilon_1\log n) > n^{b_4} \big)
    \leq n^{-b_5}
    \quad \text{for $n$ large enough},
  \end{equation}
  uniformly in $x, y \in \Z^d$. To see this we use the definition of
  $(\widehat X_n,\widehat X'_n)$ via the joint law of the cluster and
  two walkers on it to construct suitable ``corridors'' in opposite
  directions in the random environment, and force the two walks to
  walk along these corridors. Namely, denoting $r=\lfloor
  \varepsilon_1 \log n\rfloor$, and assuming without loss of
  generality that $x\cdot e_1 \le y\cdot e_1$, we require that $\omega
  (x-k e_1, k) = \omega(y+ke_1,k) = 1$ for all $k=1,\dots, r$, that
  $\xi (x-r e_1, r) = \xi (y+ r e_1, r) =1$, and that the permutations
  $\widetilde \omega $, $\widetilde \omega '$ are such $\widetilde
  \omega (x-k e_1,k)[1]=(x-(k+1)e_1,k+1)$, $\widetilde \omega' (y+k
  e_1,k)[1]=(y+(k+1)e_1,k+1)$, for all $k=0,\dots,r-1$. The
  probability that these requirements are satisfied can be easily
  bounded from below by $\delta_1^r$ for some $\delta_1=\delta_1(p, U)
  \in (0,1)$. If the requirements are satisfied, then $T^{\Sim}_k = k$
  and thus $\widehat X_k = x-k e_1$, $\widehat X'_k = y+ ke_1$ for all
  $k=1,\dots , r$. Therefore, we see that uniformly over $x, y \in
  \Z^d$
  \begin{equation}
    \label{eq:lemma:separation1:2}
    \Pr^\joint_{x,y}\big(  \|\widehat X_j-\widehat X'_j\| \geq j
      \big) \geq
    \delta_1^j.
  \end{equation}
  Thus, the probability that $(\widehat X_n)$ and $(\widehat X'_n)$ have distance
  $\varepsilon_1\log n$ after the first $\varepsilon_1\log n$ steps is
  at least $n^{-\varepsilon_1 \log (1/\delta_1)}$. If this happens, we are done,
  otherwise, we can try again by the Markov property. By the uniformity
  of the bound in (\ref{eq:lemma:separation1:2}), we have
  \begin{equation}
    \label{eq:lemma:separation1:3}
    \Pr^\joint_{x,y}\big(H(\varepsilon_1\log n) > m \varepsilon_1\log n \big)
    \leq (1-n^{-\varepsilon_1 \log(1/\delta_1)})^m
    \leq \exp(-m n^{-\varepsilon_1 \log(1/\delta_1)}).
  \end{equation}
  Now let $\varepsilon_1$ be so small that $-\varepsilon_1 \log \delta_1
  \in (0,1/2)$, pick $b_4 \in (-\varepsilon_1 \log \delta_1,1/2)$, $b_5>0$
  and set $m=b_5 n^{\varepsilon_1 \log(1/\delta_1)} \log n$ in
  (\ref{eq:lemma:separation1:3}).
  \medskip

  \noindent
  \emph{Step~2.\ }  Next we claim that for any $K_2 > 0$ we can pick a $\delta_2
  \in (0,1)$ such that for all $x,y \in \Z^d$ with $\varepsilon_1 \log n \leq
  \|x-y\| < K_2 \log n$ and $n$ large enough
  \begin{equation}
    \label{eq:lemma:separation1:4}
    \Pr^\joint_{x,y}\big( H(K_2 \log n) < h(\tfrac12\varepsilon_1 \log n) \wedge
    (K_2 \log n)^3 \big) \geq \delta_2,
  \end{equation}
  where $h$ is the stopping time defined in
  \eqref{eq:hittingtimedistouter}. To this end we couple
  $\Pr^\joint_{x,y}$ with $\Pr^\indep_{x,y}$ using
  Lemma~\ref{lemma:couplingjointregen}, in a standard way. This
  coupling implies that the left-hand side of
  \eqref{eq:lemma:separation1:4} is bounded from below by
  \begin{equation}
    \label{e:sepub}
    \Pr^\indep_{x,y}\big( H(K_2 \log n) < h (\tfrac12\varepsilon_1 \log n) \wedge
    (K_2 \log n)^3\big) - C (K_2 \log n)^3 n^{-c \varepsilon_1/2},
  \end{equation}
  where the second term is an upper bound on the probability that the coupling
  fails before the time $\min\{H(K_2\log n),h(\tfrac12 \varepsilon_1\log
  n),(K_2\log n)^3\}$. This bound follows from the fact that before this time
  the distance between $\widehat X$, $\widehat X'$ is at least $\tfrac12
  \varepsilon_1\log n$ and thus the probability that the coupling fails in one
  step is at most $\exp\{-\tfrac c2 \varepsilon_1\log n\}= n^{-c
    \varepsilon_1/2}$.

  The first term in \eqref{e:sepub} is bounded from below by a small
  constant $ \delta'$, uniformly in $n$, as follows from
  Lemma~\ref{lemma:harmonicballsd3} (with $r_1=\tfrac12\varepsilon_1
  \log n$, $r_2=K_2 \log n$) and the fact that
  $\Pr^\indep_{x,y}(\inf\{k:\|\widehat X_k-\widehat X'_k\|\geq K_2
  \log n\} > (K_2\log n)^3) \to 0$ as $n\to\infty$. The latter
  assertion holds because by the invariance principle we have for a
  $d$-dimensional Brownian motion
  \begin{multline*}
    \Pr^\indep_{x,y}  \left (\inf\left \{k:\|\widehat X_k-\widehat
        X'_k\|\geq K_2 \log n \right\} > (K_2\log n)^3\right) \\
     \approx \Pr \left(\inf\left \{ t: \| B_t\| \ge \frac{K_2 \log
          n}{\sqrt{2n}} \right\} > \frac{(K_2\log n)^3}{n} \right) \le
    \frac{n}{(K_2 \log n)^3} \frac{(K_2 \log n )^2}{\sqrt{2dn}},
  \end{multline*}
  where the last inequality follows from the Markov inequality and the
  fact that the expected exit time of a $d$-dimensional Brownian from
  a ball of radius $r$ is bounded by $r^2/d$. As the second term in
  \eqref{e:sepub} converges to $0$ as $n\to\infty$, the proof of
  \eqref{eq:lemma:separation1:4} is completed.

  \medskip

  \noindent
  \emph{Step~3.\ }By repeating the argument from Step~1 and using
  (\ref{eq:lemma:separation1:4}) from Step~2, we see that we can choose a
  (large) $K_3$, $b_6 \in (b_4,1/2)$ such that uniformly in $x,y \in \Z^d$
  \begin{equation}
    \label{eq:lemma:separation1:5}
    \Pr^\joint_{x,y}\big(H(K_3 \log n) \leq n^{b_6}  \big) \geq \delta_3 > 0
    \quad \text{for $n$ large enough}.
  \end{equation}

  The previous steps work for all $d\ge 2$. In the next steps we shall
  treat separately the cases $d=2$ and $d\ge 3$. We start with the
  case $d\ge 3$.

  \medskip
  \noindent
  \emph{Step~4 ($d\ge 3$).} Arguing as in Step~2, we can find $b_1 \in (0,1/6)$
  such that for all $x,y \in \Z^d$ with
  $K_3 \log n \leq \|x-y\| < n^{b_1}$ and $n$ large enough
  \begin{equation}
    \label{eq:lemma:separation1:6}
    \Pr^\joint_{x,y}\big( H(n^{b_1}) < h (\tfrac12 K_3 \log n) \wedge n^{3
      b_1}\big) \geq \delta_4 > 0.
  \end{equation}
  \smallskip

  \noindent
  \emph{Step~5 ($d\ge 3$).} Now we recycle the argument from Step~1 as
  follows: Wait until $(\widehat X_n)$ and $(\widehat X'_n)$ have
  reached distance at least $K_3 \log n$ or stop if the waiting time
  exceeds $n^{b_6}$. Then, let $(\widehat X_n,\widehat X'_n)$ run
  until they have either reached distance $n^{b_1}$ or have taken
  (another) $n^{3b_1}$ steps without reaching that distance. Note that
  by construction, such an attempt takes at most $n^{b_6}+n^{3b_1}$
  time steps, and by (\ref{eq:lemma:separation1:5}) and
  (\ref{eq:lemma:separation1:6}), with probability at least $\delta_3
  \delta_4$ leads to a separation of $n^{b_1}$, as required. If it
  does not, we start afresh, using the Markov property and the
  uniformity of (\ref{eq:lemma:separation1:5},
  \ref{eq:lemma:separation1:6}) in their respective initial
  conditions.

  Pick $b_2, b_7$ such that $b_6 \vee 3b_1 < b_7 < b_2 < 1/2$ (in
  particular, $n^{b_7} \geq n^{b_6}+n^{3b_1}$ for $n$ large enough),
  put $b_4 \coloneqq b_2-b_7$, $b_3 \coloneqq -\log(1-\delta_3
  \delta_4)$. The probability that the first $n^{b_4}$ attempts all
  fail is bounded above by $(1-\delta_3\delta_4)^{n^{b_4}} = \exp(-b_3
  n^{b_4})$, and by construction these first $n^{b_4}$ attempts take
  at most $n^{b_4}(n^{b_6}+n^{3b_1}) \leq n^{b_4+b_7}=n^{b_2}$ time
  steps, which proves the lemma for $d\ge 3$. \medskip

  \noindent\emph{Step 6 ($d=2$).} In $d=2$, the relation
  \eqref{eq:lemma:separation1:6} should be replaced by
  \begin{equation}
    \label{eq:lemma:separation1:7}
    \Pr^\joint_{x,y}\big( H(n^{b_1}) < h(\tfrac12 K_3 \log n) \wedge n^{3 b_1}
    \big) \geq \frac {\log 2}{ b_1 \log n},
  \end{equation}
  which can be shown using the same argument as in Step~2,
  together with Lemma~\ref{lemma:harmonicballsd3}. The argument in Step~5
  is then analogous: The probability that the walks separate by distance
  $n^{b_1}$ in the first $n^{b_6}+n^{3b_1}$ steps is bounded from below
  by $\delta_3 \log2/ (b_1 \log n)$. Fixing $b_4=(b_2-b_7)/2$, we see
  that the probability that the first $n^{2b_4}$ attempts fail is smaller than
  $(1-\delta_3 \log 2/(b_1 \log n))^{n^{2b_4}}\le \exp(-b_3 n^{b_4})$, for
  some small $b_3$ and $n$ large enough. These first $n^{2b_4}$ attempts
  take at most $n^{2b_4}(n^{b_4}+n^{3b_1})\leq n^{b_2}$ time steps, as
  required.
\end{proof}

Using the above separation lemma, we now construct a coupling of
$\Pr^\joint_{\mathbf 0,\mathbf 0}$ and $\Pr^\indep_{\mathbf 0,\mathbf 0}$ so
that the increments differ only in few steps.

\begin{lemma}[Coupling of dependent and independent $\Xi$-chains]
  \label{l:hatcoupling} For any $d\ge 2$ there is a Markov chain
  $(\Xi^\joint_n,\Xi^\indep_n)_n$ with state space $(\mathbbm F\times
  \mathbbm F\times \Z^d \times \Z^d)^2$ such that \, $\Xi^\joint$ is
  $\Pr^\joint_{\mathbf 0, \mathbf 0}$- and \, $\Xi^\indep$ is
  $\Pr^\indep_{\mathbf 0,\mathbf 0}$-distributed. Furthermore, writing
  $\Xi^\joint_n = (\alpha^\joint_n,\alpha '^\joint_n,\widehat
  X^\joint_n,\widehat X'^\joint_n)$ and $\Xi^\indep_n =
  (\alpha^\indep_n,\alpha '^\indep_n,\widehat X^\indep_n,\widehat
  X'^\indep_n)$, there exist $b_6>0$ and $b_5\in (0,1/2)$ so that for
  all $N$ large enough,
  \begin{equation}
    \label{e:Xicoupling}
    \Pr\big(\#\{k\le N: (\alpha^\joint_k,\alpha '^\joint_k)\neq
        (\alpha^\indep_k,\alpha '^\indep_k)\}\ge N^{b_5}\big)\le N^{-b_6}.
  \end{equation}
\end{lemma}
\begin{proof}
  First observe that under $\Pr^{\indep}_{\mathbf 0,\mathbf 0}$ the sequences
  $\alpha^\indep$ and $\alpha'^\indep$ are independent and i.i.d., as the
  law of the increments up to the next regeneration do not depend on the
  positions $\widehat X^\indep$, $\widehat X'^\indep$.

  Using this observation, the construction of the coupling in $d\ge 3$ is
  easy. We first run the Markov chains $\Xi^\indep$,
  $\Xi^\joint$ independently according to their prescribed laws up to the first
  time $T$ when $\|\widehat X^\joint_T-\widehat X'^\joint_T\|\ge N^{b_1}$, for
  $b_1$ as in Lemma~\ref{lemma:separation1}. Let $A_1$ be the event
  $A_1=\{T\le N^{b_2}\}$. According to Lemma~\ref{lemma:separation1},
  $\Pr(A_1)\ge 1-e^{-b_3 N^b_4}$.

  Let, for some large $K$, $A_2$ be the event
  \begin{equation}
    A_2\coloneqq\{\|\widehat X^\joint_k-\widehat X'^\joint_k\|\ge K \log N
      \text{ for all } T\le k \le N\}.
  \end{equation}
  By comparing $\widehat X^\joint$ with $\widehat X^\indep$, as in the proof of
  Lemma~\ref{lemma:separation1}, and using elementary properties of the
  random walk and the fact
  $\|\widehat X^\joint_T-\widehat X'^\joint_T\|\ge N^{b_1}$, it is easy to see
  that $\mathbb P[A_2^c]\le N^{-c}$ for some $c> 0$

  On $A_1\cap A_2$, couple $\Xi^\indep$ and $\Xi^\joint$ so that
  $(\alpha^\joint_k,\alpha '^\joint_k) = (\alpha^\indep_k,\alpha '^\indep_k)$
  when possible for all $k\in [T,N]$. Using
  Lemma~\ref{lemma:couplingjointregen} and the observation in the first
  paragraph of this proof, the probability that this does not
  occur is at most $N e^{-cK\log N}$. On $A_1^c\cup A_2^c$, run
  $\Xi^\indep$, $\Xi^\joint$ independently up to time $N$. Obviously, this
  coupling satisfies \eqref{e:Xicoupling}.

  In dimension $d=2$ the situation is slightly more complicated, as
  the event $A_2$ has a small probability. We need to decompose the
  trajectory of $\Xi^\joint$ into excursions. For large constants
  $K,K'$, we define stopping times $\mathcal R_i$, $\mathcal D_i$,
  $\mathcal U$ by $\mathcal R_0= 0$, and for $i\ge 1$
  \begin{equation}
    \begin{split}
      \mathcal D_i&=\inf\{k\ge \mathcal R_{i-1}:\|\widehat X^\joint_k-\widehat X'^\joint_k\|\ge
        N^{b_1}\},\\
      \mathcal R_i&=\inf\{k\ge \mathcal D_i:\|\widehat X^\joint_k-\widehat X'^\joint_k\|\le
        K\log N\},\\
      \mathcal U &= \inf\{k\ge 0: \|\widehat X^\joint_k-\widehat X'^\joint_k\|\ge
        K'N\}.
    \end{split}
  \end{equation}
  We set $J$ to be the unique integer such that
  $\mathcal D_J\le \mathcal U\le \mathcal R_J$. The
  random variable $J$ has a geometric distribution with parameter
  converging to $b_1$ as $N\to\infty$, as can be easily shown by
  comparing $\Xi^\joint$ with $\Xi^\indep$ as in the proof of
  Lemma~\ref{lemma:separation1} and applying
  Lemma~\ref{lemma:harmonicballsd3}.
  Therefore, $\mathbb P(J\ge \log N)\le N^{-c}$ for a $c>0$.
  Applying Lemma~\ref{lemma:separation1}, we get
  \begin{equation}
    \label{e:oneshort}
    \mathbb P(\mathcal D_i-\mathcal R_{i-1}\ge N^{b_2}) \le e^{-b_3 N^{b_4}}.
  \end{equation}
  Combining these two facts we obtain
  \begin{equation}
    \label{e:insideshort}
    \mathbb P\Big(\sum_{i=1}^J \mathcal D_i-\mathcal R_{i-1}\ge N^{b_2}\log N\Big)\le
    N^{-c}.
  \end{equation}
  On the other hand, comparing again $\Xi^\joint$ with $\Xi^\indep$,
  and using simple random walk large deviation estimates, for $K'$
  large enough,
  \begin{equation}
    \label{e:Ubig}
    \mathbb P(\mathcal U\le N)\le e^{-c N}\le N^{-c}.
  \end{equation}
  Inequalities \eqref{e:insideshort} and \eqref{e:Ubig} yield
  \begin{equation}
    \begin{split}
      \mathbb P&\left(\#\{k\le N:\|\widehat X^\joint_k-\widehat X'^\joint_k\|\le K\log N\}\ge N^{b_2}\log N\right)
      \\&\le
      \mathbb P\left(\sum_{i:\mathcal R_{i-1}\le N} \mathcal D_i-\mathcal R_{i-1}\ge N^{b_2}\log N\right)\le
      N^{-c}.
    \end{split}
  \end{equation}
  If the event on the left-hand side of the last display does not occur,
  we can with probability at least $1- N^{-c}$ couple $\Xi^\joint$
  and $\Xi^\indep$ so that
  $(\alpha^\joint_k,\alpha '^\joint_k) = (\alpha^\indep_k,\alpha '^\indep_k)$
  for all $k$ satisfying $\mathcal D_i\le k \le \mathcal R_i$
  for some $i$,
  using Lemma~\ref{lemma:couplingjointregen} (and noting that
  under $\Psi^\indep$, the law of $(\alpha^\indep_1,\alpha '^\indep_1)$
  does not depend on the starting point). Taking $b_5$ satisfying $b_2<b_5<1/2$,
  \eqref{e:Xicoupling} is proved for $d=2$.
\end{proof}

\begin{lemma}
  \label{lemma:couplerrord3lip}
  Let $d \geq 2$. Recall that $\widetilde X$, $\widetilde X'$ are read from $\Xi$
  as in Remark~\ref{rem:Xichain}. Then,
  there exist $b, C >0$ such that for every pair of bounded Lipschitz
  functions $f, g : \R^d \to \R$
  \begin{equation}
    \label{eq:couplerrord3lip}
    \begin{split}
      \Big| & \E^\joint_{\mathbf 0, \mathbf 0}\big[
      f(\widetilde{X}_n/\sqrt{n})g(\widetilde{X}'_n/\sqrt{n}) \big] -
      \E^\indep_{\mathbf 0, \mathbf 0}\big[
      f(\widetilde{X}_n/\sqrt{n})g(\widetilde{X}'_n/\sqrt{n})\big] \Big| \\
      & \leq C \big(1+\|f\|_\infty + L_f\big) \big(1+\|g\|_\infty + L_g\big)
      n^{-b},
    \end{split}
  \end{equation}
  where $L_f \coloneqq \sup_{x \neq y} |f(y)-f(x)|/\|y-x\|$ and $L_g$ are the Lipschitz
  constants of $f$ and $g$.
\end{lemma}
\begin{proof}
  We use the coupling constructed in the last lemma. Let $\mathcal I$
  be the complement of the set of indices appearing in
  \eqref{e:Xicoupling},
  \begin{equation}
    \mathcal I =\{k\le n: (\alpha^\joint_k,\alpha '^\joint_k) =
    (\alpha^\indep_k,\alpha '^\indep_k)\},
  \end{equation}
  and set $\mathcal I^c = \{0,\dots,n\}\setminus \mathcal I$. By the last lemma,
  \begin{equation}
    \label{e:Icsmall}
    \mathbb P\left(\abs{\mathcal I^c}\ge n^{b_5}\right)\le n^{-b_6}.
  \end{equation}

  We now read the processes $\widetilde X^\joint$, $\widetilde
  X'^\joint$, $\widetilde X^\indep$, $\widetilde X'^\indep$ out of
  $\Xi^\joint$, $\Xi^\indep$. Recall notation
  \eqref{eq:defRi}--\eqref{eq:defhatXell}. We use the additional
  superscript $\cdot^\indep$ or $\cdot^\joint$ in this notation in the
  obvious way, and write $T^\joint$, $T^\indep$ for $T^\Sim$
  corresponding to those processes.

  By Lemma~\ref{lemma:jointrenewaltails}, and standard large deviation
  estimates, there is a large constant $K$ so that
  \begin{equation}
    \mathbb P (T^\joint_n\ge Kn) \le e^{- cn}
  \end{equation}
  and, using \eqref{e:Icsmall} as well,
  \begin{equation}
    \label{e:badtimes}
    \mathbb P\left(\sum_{i\in \mathcal I^c}
      \bigl(T^\joint_i-T^\joint_{i-1}\bigr) \ge
      Kn^{b_5}\right)\le n^{-b_6},
  \end{equation}
  and similarly for $T^\indep$.

  We now consider the process
  $\widetilde X^\indep$ and define two sets
  \begin{equation}
    \begin{split}
      \mathcal G^\indep & \coloneqq \{1\le k\le J^\indep_n:[T_{k-1},T_k]\subset
        [T^\indep_{j-1},T^\indep_j] \text{ for some } j\in \mathcal I\}\\
      \mathcal B^\indep &\coloneqq \{1, \dots,  J^\indep_n\}\setminus \mathcal G^\indep.
    \end{split}
  \end{equation}
  We define $\mathcal G'^\indep$, $\mathcal G^\joint$, \dots
  analogously. Note that $\mathcal G$'s are the sets of indices of
  those increments that `occur during coupled periods of $\Xi$'s'.
  More precisely, by the coupling construction, there is an
  order-preserving bijection $\phi $ of $\mathcal G^\indep$ and
  $\mathcal G^\joint$ so that for every $j \in \mathcal G^\indep$
  (using the notation introduced before \eqref{eq:deftildeXm}),
  \begin{equation}
    Y^\indep_j = \widetilde X^\indep_j - \widetilde X^\indep_{j-1}= \widetilde X^\joint_{\phi (j)} - \widetilde
    X^\joint_{\phi(j)-1} = Y^\joint_{\phi(j)}.
  \end{equation}
  Therefore, setting $[n]\coloneqq\{1,\dots,n\}$, and $\mathcal R_{[n]} \coloneqq \mathcal R
  \cap [n]$ for any set $\mathcal R$, we can write
  \begin{equation}
    \begin{split}
      \label{e:crazydec}
      \widetilde X^\indep_n
      &= \sum_{i=1}^n Y^\indep_i = \sum_{i\in \mathcal G^\indep_{[n]}} Y^\indep_i
      +\sum_{i\in \mathcal B^\indep_{[n]}} Y^\indep_i
      \\ & =\sum_{j\in \phi(\mathcal
        G^\indep_{[n]})} Y^\joint_j +\sum_{i\in \mathcal B^\indep_{[n]}}
      Y^\indep_i
      \\
      &= \widetilde X^\joint_n +
      \sum_{j\in \phi(\mathcal G^\indep_{[n]})\setminus[n]} Y^\joint_j
      -\sum_{j\in \mathcal G^\joint_{[n]} \setminus \phi (\mathcal G^\indep_{[n]})} Y^\joint_j
      +\sum_{i\in \mathcal B^\indep_{[n]}} Y^\indep_i -\sum_{i\in \mathcal
        B^\joint_{[n]}} Y^\joint_i.
    \end{split}
  \end{equation}
  Similar claims hold for the processes with primes.

  Inequality \eqref{e:badtimes} implies that
  \begin{equation}
    \label{e:Bsmalls}
    \mathbb P\left(|\mathcal B^\indep|\vee | \mathcal B'^\indep | \vee
      |\mathcal B^\joint|\vee | \mathcal B'^\joint | \ge K n^{b_5}\right)\le
    n^{-b_6}.
  \end{equation}
  When the complement of event in \eqref{e:Bsmalls} holds, then all
  four sums on the right-hand side of \eqref{e:crazydec} have length
  at most $2Kn^{b_5}$. Since all relevant increments have (uniformly
  bounded) exponential tails, by standard large deviation estimates we
  can choose $K'$ large so that the probability that the absolute
  value of these four sums exceed $K' n^{b_5}$ is at most $e^{-c
    n^{b_5}}$, for all $n$ large enough.

  Putting all these claims together we see that there is an event $A$
  satisfying $\mathbb P(A^c)\le n^{-b'}$ with $b'>0$, so that on $A$
  \begin{equation}
    \label{e:Xclose}
      |\widetilde X^\indep_n - \widetilde X^\joint_n|\le K 'n^{b_5}
      \qquad\text{and}\qquad
      |\widetilde X'^\indep_n - \widetilde X'^\joint_n|\le K 'n^{b_5}.
  \end{equation}
  Therefore,
  \begin{equation}
    \begin{split}
      &\Big|  \E^\joint_{\mathbf 0, \mathbf 0}\big[ f(\widetilde{X}_n/\sqrt{n})g(\widetilde{X}'_n/\sqrt{n}) \big]
      - \E^\indep_{\mathbf 0,\mathbf 0}\big[ f(\widetilde{X}_n/\sqrt{n})g(\widetilde{X}'_n/\sqrt{n})\big] \Big|
      \\& =
       \Big| \E\big[  f(\widetilde X^\joint_n/\sqrt{n})g(\widetilde X'^\joint_n/\sqrt{n}) -
        f(\widetilde X^\indep_n/\sqrt{n})g(\widetilde X'^\indep_n/\sqrt{n})
        \big]\Big|  \\
      &\leq   2 \Pr(A^c)\|f\|_\infty\|g\|_\infty
      \\&\quad + \E\Big[ \indset{A}
        \big| f(\widetilde X^\joint_n/\sqrt{n})g(\widetilde X'^\joint_n/\sqrt{n}) -
        f(\widetilde X^\indep_n/\sqrt{n})g(\widetilde
          X'^\indep_n/\sqrt{n})\big|\Big].
    \end{split}
  \end{equation}
  Observing that
  \begin{equation} \label{eq:22}
    \big| f(x)g(y) - f(x')g(y') \big| \leq
    \|g\|_\infty L_f \|x-x'\| + \|f\|_\infty L_g \|y-y'\|,
  \end{equation}
  and $b_5 < 1/2$, using \eqref{e:Xclose}, this implies the claim with
  $b\coloneqq b' \wedge (1/2-b_5)$.
\end{proof}

\begin{remark}[Functional limit theorem] \label{rem:inv-principle-2}
  For the quenched functional limit theorem in the case $d \ge 2$ we
  need also a functional version of the above theorem. The proof above
  can be adapted to show the analogue of \eqref{e:Xclose}, i.e.\ that
  on some event $\widetilde A$, satisfying $\Pr(\widetilde A^c) \le
  n^{-b'}$ with $b'>0$, we have
   \begin{align*}
     \sup_{t \in [0,1]} |\widetilde X^\indep_{[nt]} - \widetilde
     X^\joint_{[nt]}|\le  \widetilde K n^{b_5}   \qquad \text{and} \qquad
     \sup_{t \in [0,1]} |\widetilde X'^\indep_{[nt]} - \widetilde
     X'^\joint_{[nt]}|\le  \widetilde K n^{b_5}.
  \end{align*}
\end{remark}

\begin{proof}[Proof of Lemma~\ref{lemma:qCLTars}, case $d \ge 2$]
  Lemma~\ref{lemma:qCLTars} follows now easily from
  Lemma~\ref{lemma:couplerrord3lip} together with standard Berry-Esseen
  type estimates for the term
  $\E\big[ f(\widetilde{X}_n/\sqrt{n})g(\widetilde{X}'_n/\sqrt{n})\big]$
  appearing in \eqref{eq:couplerrord3lip} there.
  \smallskip

  The case $d=1$ requires a somewhat different approach and is given
  in Section~\ref{sec:cased1}.
\end{proof}

\begin{lemma}
  \label{lemma:triumphofmoderatedeviations}
  Assume that for some $b>1$, and any bounded Lipschitz function $f : \R^d \to \R$
  \begin{align}
    E_\omega\big[f(\widetilde{X}_{k^b}/k^{b/2}) \big] \mathop{\longrightarrow}_{k\to\infty}
    \widetilde{\Phi}(f) \quad
    \text{for $\Pr(\,\cdot\, | B_0)$-a.a.\ $\omega$},
  \end{align}
  where $\widetilde\Phi$ is some non-trivial centred $d$-dimensional normal law.
  Then we have for any bounded Lipschitz function $f$
  \begin{align}
    E_\omega\big[f(\widetilde{X}_{m}/m^{1/2})\big] \mathop{\longrightarrow}_{m\to\infty}
    \widetilde{\Phi}(f) \quad
    \text{for $\Pr(\,\cdot\, | B_0)$-a.a.\ $\omega$}.
  \end{align}
\end{lemma}
\begin{proof}
  Note that under $\Pr(\cdot|B_0)$, the increments
  $Y_i=\widetilde{X}_i-\widetilde{X}_{i-1}$ are i.i.d., centred and satisfy
  $\E[\exp(\lambda\cdot Y_1)|B_0] < \infty$ for all $\lambda$ in some
  neighbourhood of the origin. Let $(2-b) \vee 0 < \varepsilon < 1$. By a
  moderate deviation principle (e.g. Thm.~3.7.1 in \citet{MR2571413} where we
  set $n=b k^{b-1}$, $a_n=b k^{-1+\varepsilon}$, then $a_n\to0$, $n a_n
  \to\infty$, $\sqrt{a_n/n}=k^{-(b-\varepsilon)/2}$), we have for any $\delta>0$
  \begin{align}
    \label{eq:modedevbd1}
    \Pr\Big( \max_{1 \leq i \leq b k^{b-1}} \frac{|\widetilde{X}_i|}{k^{(b-\varepsilon)/2}} \geq \delta
    \Big| B_0 \Big)
    \leq C k^{b-1} \exp(-c k^\eta)
  \end{align}
  for some $C, c, \eta > 0$ ($\eta$ can be chosen in $(0,1-\varepsilon)$)
  and all $k \in \N$.
  Since the right-hand side of (\ref{eq:modedevbd1}) is summable in $k$,
  noting that $k^b - (k-1)^b \leq bk^{b-1}$,
  we obtain by Borel-Cantelli
  \begin{align}
    \limsup_{k\to\infty}
    \max_{(k-1)^b \leq \ell \leq k^b} \frac{|\widetilde{X}_\ell - \widetilde{X}_{(k-1)^b}|}{k^{b/2}}
    \leq
    \limsup_{k\to\infty}
    \max_{(k-1)^b \leq \ell \leq k^b} \frac{|\widetilde{X}_\ell - \widetilde{X}_{(k-1)^b}|}{k^{(b-\varepsilon)/2}}
    = 0
  \end{align}
  for $\Pr(\,\cdot\, | B_0)$-a.a.\ $\omega$.
\end{proof}

  \subsection{Proof of Theorem~\ref{thm:quenchedCLT}}

  The idea is to use the variance control provided by
  Lemma~\ref{lemma:couplerrord3lip} to obtain the quenched CLT.
  This approach seems to have appeared in the literature on
  random walk in random environments for the first time in
  \cite{BolthausenSznitman:2002}.
  \smallskip
  Let $f : \R^d\to\R$ be bounded and Lipschitz, $b'> 1/b \vee 1$ with $b$ from
  Lemma~\ref{lemma:qCLTars}.
  By \eqref{eq:L2regenCLT} and Markov's inequality,
  \begin{align}
    \Pr\left( \Big| E_\omega\big[f\big(\widetilde{X}_{[n^{b'}]}/\sqrt{[n^{b'}]}\big)\big]
        -\widetilde{\Phi}(f) \Big| > \varepsilon \right)
      \leq C_f \varepsilon^{-2} n^{-b' b},
  \end{align}
  which is summable, hence
  \begin{align}
    E_\omega\big[f\big(\widetilde{X}_{[n^{b'}]}/\sqrt{[n^{b'}]}\big)\big]
    \to \widetilde{\Phi}(f) \quad \text{a.s. as} \;\; n \to \infty
  \end{align}
  by Borel-Cantelli.
  Now Lemma~\ref{lemma:triumphofmoderatedeviations} yields
  \begin{align}
    E_\omega\big[f(\widetilde{X}_{m}/m^{1/2})\big]
    \mathop{\longrightarrow}_{m\to\infty} \widetilde{\Phi}(f) \quad
    \text{for $\Pr(\,\cdot\, | B_0)$-a.a.\ $\omega$}.
  \end{align}

  Put
  \begin{align}
    V_n \coloneqq \max \{ m \in \Z_+ : T_m \leq n \}.
  \end{align}
  We have ${V_n}/{n} \to 1/\E[\tau_1]$ a.s.\ as $n\to\infty$ by
  classical renewal theory, in fact
  \begin{align}
    \label{eq:LILrenewal}
    \limsup_{n\to\infty} \frac{\big| V_n-n/\E[\tau_1]\big|}{\sqrt{n \log\log n}}
    < \infty \quad \text{a.s.}
  \end{align}
  (see, e.g. Thm.~III.11.1 in \citet {MR916870}).
 Since the differences of the $T_m$ are i.i.d.\ with exponential tail bounds,
  \begin{align}
    \label{eq:maxrenewalgaplogarithmic}
    \limsup_{n\to\infty} \frac{\max_{j \leq n}\{j - T_{V_j}\}}{\log n} < \infty
    \quad \text{a.s.}
  \end{align}

  Recall that $X_{T_{V_n}}=\widetilde{X}_{V_n}$. Now
  \begin{align}
    P_\omega\big( \norm{X_n-\widetilde{X}_{V_n}} \geq \log^2n \, \big)
    \to 0 \qquad \text{a.s.}
  \end{align}
  by \eqref{eq:maxrenewalgaplogarithmic} and the fact that $X$ has
  bounded increments. Furthermore, for any $\varepsilon>0$
  \begin{align}
    P_\omega\big( | V_n-n/\E[\tau_1] | \geq n^{1/2+\varepsilon}\, \big)
    \to 0 \qquad \text{a.s.}
  \end{align}
  by \eqref {eq:LILrenewal}.

  Moreover, there exist $\beta \in (1/2,1)$ and $\gamma \in (\beta/2,1/2)$
  such that for any $\theta \geq 0$,
  \begin{align}
    \label{eq:maxincrrw1}
    \limsup_{n\to\infty} \sup_{|k-\theta n| \leq n^\beta}
    \frac{\big| \widetilde{X}_k - \widetilde{X}_{[\theta n]}\big|}{n^\gamma}
    \to 0 \qquad \text{a.s.}
  \end{align}
  To prove \eqref{eq:maxincrrw1} note that we can take w.l.o.g.\
  $\theta=0$. By Doob's $L^6$-inequality, for any $\varepsilon>0$
  \begin{align}
    \label{eq:maxincrrw1b}
    \Pr\Big( \sup_{k \leq n^\beta} \big| \widetilde{X}_k - \widetilde{X}_0 \big|
    > \varepsilon n^\gamma \Big)
    & \leq
    \varepsilon^{-6} n^{-6\gamma}
    \E\Big[ \norm{\widetilde{X}_{n^\beta} - \widetilde{X}_0}^6 \Big]
    \leq \mathrm{Const.} \, \varepsilon^{-6} n^{3\beta-6\gamma}
  \end{align}
  (note that $\E_0[\|S_k\|^6] \leq C k^3$ for a random walk $(S_k)$ whose
  increments are centred and have bounded $6$th moments), thus we can choose
  $\beta>1/2>\gamma$ and both sufficiently close to $1/2$ that
  $3\beta-6\gamma<-1$ and the right-hand side of \eqref{eq:maxincrrw1b} becomes
  summable in $n$. The usual Borel-Cantelli argument allows to conclude
  \eqref{eq:maxincrrw1}.

  Now write
  \begin{align}
    \frac{X_n}{\sqrt{n}} & =
    \frac{X_n-\widetilde{X}_{V_n}}{\sqrt{n}}
    + \frac{\widetilde{X}_{V_n}-\widetilde{X}_{[n/\E[\tau_1]]}}{\sqrt{n}}
    + \frac{\widetilde{X}_{[n/\E[\tau_1]]}}{\sqrt{n/\E[\tau_1]}}
      \sqrt{1/\E[\tau_1]}
  \end{align}
  and let $\Phi$ be defined by
  $\Phi(f)\coloneqq\widetilde{\Phi}\big(f((\E[\tau_1])^{-1/2} \, \cdot)\big)$,
  i.e.\ ${\Phi}$ is the image measure of $\widetilde{\Phi}$ under $x \mapsto
  x/\sqrt{\E[\tau_1]}$. Then
  \begin{align*}
    & \left| E_\omega\big[f(X_{n}/n^{1/2}) \big] - \Phi(f) \right| \\
    & \leq
    2 \norm{f}_\infty \Big(
    P_\omega\big( \norm{X_n-\widetilde{X}_{V_n}} \geq \log^2n \,\big)
    + P_\omega\big( | V_n-n/\E[\tau_1] | \geq n^{\beta}\, \big) \\
    & \hspace{5em} {} +
    P_\omega\big( \sup\nolimits_{|k-n/\E[\tau_1]| \leq n^\beta}
    \big| \widetilde{X}_k - \widetilde{X}_{[n/\E[\tau_1]]}\big| > n^\gamma \big)
    \Big) \\
    & \quad
    + L_f \big( \log^2 n/\sqrt{n} + n^{\gamma-1/2} \big) \\
    & \quad
    + \left| E_\omega\big[f\big((\E[\tau_1])^{-1/2} \times
      \widetilde{X}_{[n/\E[\tau_1]]}/\sqrt{n/\E[\tau_1]}\big)\big]
      - \widetilde{\Phi}\big(f((\E[\tau_1])^{-1/2} \, \cdot)\big) \right| \\
    & \longrightarrow 0 \qquad \text{a.s.\ as $n\to\infty$}.
  \end{align*}
  This proves \eqref{eq:quenchedCLT} for bounded Lipschitz functions
  $f$, which in particular implies that the laws of $(X_n/\sqrt{n})$
  under $P_\omega$ are tight, for almost all $\omega$.
  Finally, note that a general continuous bounded $f$ can be approximated
  by bounded Lipschitz functions in a locally uniform way.
  \hfill \qed

\subsection{Proof of Lemma~\ref{lemma:qCLTars}, case $d=1$}
\label{sec:cased1}

Here we prove Lemma~\ref{lemma:qCLTars} in the case $d=1$. In the
cases $d \ge 2$ our proof made substantial use of the fact that two
$d$-dimensional random walks typically do not spend much time near
each other: The ``approximate collision time'' during the first $n$
steps is typically $O(\log n$) in $d=2$ and $O(1)$ in $d \ge 3$. Thus,
we could couple two walks on the same cluster with two walks on
independent copies of the cluster and the error incurred becomes
negligible when dividing by $\sqrt{n}$ in the CLT (see the proofs of
Lemmas\ \ref{l:hatcoupling} and \ref{lemma:couplerrord3lip}). An
analogous strategy can not be ``na\"ively'' implemented in $d=1$
because now we expect typically $O(\sqrt{n})$ hits of the two walks,
so using simple ``worst case bounds'' in the region where the two
walkers are close would yield an (overly pessimistic) error term that
does not vanish upon dividing by $\sqrt{n}$. Instead, we now use a
martingale decomposition for the dynamics of $(\widehat X_n, \widehat
X_n')_n$ under $\widehat{\Psi}^\joint$, combined with a quantitative
martingale CLT from \cite{Rackauskas:1995} to estimate the Kantorovich
distance from the bivariate normal.

\smallskip

Consider as a toy example a Markov chain on $\Z$ that behaves as a
symmetric simple random walk as long as $X_n \ne 0$. Upon hitting $0$
the chain stays there for some random time and leaves this state
symmetrically. If the distribution of the time the chain spends in $0$
is suitably controlled, one can prove a central limit theorem with
non-trivial limit by using a martingale central limit theorem. A
similar argument works for our two random walks (observed along joint
renewal times) where in this case ``being at zero'' corresponds to the
event that the two walks are closer together than $K \log n$ for some
appropriate constant $K$. This is the ``black box'' region in which we
cannot couple the two walks to independent copies. If they are more
than $n^{b}$ for small $b$ away from each other then we can couple
with very good control of the error
(cf.~Lemma~\ref{lemma:couplingjointregen}). Then we make use of the
symmetries of the model and the fact that in $d=1$, the walks
$\widehat{X}$ and $\widehat{X}'$ have many overcrossings to verify
that the error terms stemming from times inside the black box up to
time $n$ are in fact $o(\sqrt{n})$ (in a suitably quantitative sense).

\medskip

Let $(\widehat X_n, \widehat X_n')_n$ be a pair of walks in $d=1$ on
the same cluster observed along joint renewal times as in
(\ref{eq:defhatXell}), which is in itself a Markov chain when
averaging over the cluster, and let
$\widehat\Psi^\joint((x,x'),(y,y'))$ be its transition probability, as
defined in (\ref{eq:PSihatjoint}) in Remark~\ref{rem:Xichain}. We
write $\widehat{\mathcal{F}}_n \coloneqq \sigma\big(\widehat X_i,
\widehat X_i', 0 \le i \le n\big)$ for the canonical filtration of
this chain.

Furthermore set
\begin{align*}
  \phi_1(x,x') &  \coloneqq \sum_{y, y'} (y-x)\widehat\Psi^\joint((x,x'),(y,y')), \\
  \phi_2(x,x') &  \coloneqq \sum_{y, y'}
  (y'-x')\widehat\Psi^\joint((x,x'),(y,y')), \\
  \phi_{11}(x,x') &  \coloneqq \sum_{y, y'}
  \big(y-x-\phi_1(x,x')\big)^2\widehat\Psi^\joint((x,x'),(y,y')), \\
  \phi_{22}(x,x') &  \coloneqq \sum_{y, y'}
  \big(y'-x'-\phi_2(x,x')\big)^2\widehat\Psi^\joint((x,x'),(y,y')), \\
  \phi_{12}(x,x') &  \coloneqq \sum_{y, y'}
  \big(y-x-\phi_1(x,x')\big)\big(y'-x'-\phi_2(x,x')\big)
  \widehat\Psi^\joint((x,x'),(y,y')).
\end{align*}
Note that by Lemma~\ref{lemma:jointrenewaltails} these are uniformly bounded, i.e.\
\begin{equation}
\label{eq:phiunifbd}
C_\phi \coloneqq \norm{\phi_1}_\infty \vee \norm{\phi_2}_\infty
\vee \norm{\phi_{11}}_\infty \vee \norm{\phi_{12}}_\infty
\vee \norm{\phi_{22}}_\infty < \infty.
\end{equation}
Define
\begin{align}
  A_n^{(1)} & \coloneqq \sum_{j=0}^{n-1} \phi_1(\widehat X_j,\widehat X_j'), \quad
  A_n^{(2)} \coloneqq \sum_{j=0}^{n-1} \phi_2(\widehat X_j,\widehat X_j'), \notag \\
  A_n^{(11)} & \coloneqq \sum_{j=0}^{n-1} \phi_{11}(\widehat X_j,\widehat X_j'), \quad
  A_n^{(22)} \coloneqq \sum_{j=0}^{n-1} \phi_{22}(\widehat X_j,\widehat X_j'), \quad
  A_n^{(12)} \coloneqq \sum_{j=0}^{n-1} \phi_{12}(\widehat X_j,\widehat X_j'), \notag \\
  \label{defMnMnhat}
  M_n & \coloneqq \widehat X_n - A_n^{(1)},
  \qquad M_n' \coloneqq \widehat X_n' - A_n^{(2)}.
\end{align}
Then, $(M_n)$, $(M_n')$, $(M_n^2-A_n^{(11)})$, $({M_n'}^2-A_n^{(22)})$
and $(M_n M_n' - A_n^{(12)})$ are martingales whose increments have
exponential tails (by Lemma~\ref{lemma:jointrenewaltails}, which in
particular shows that exponential tail bounds can be chosen uniformly
in $n$).

Write $\widehat{\sigma}^2\coloneqq\sum_{y, y'} y^2\widehat\Psi^\indep((0,0),(y,y'))$
for the variance of a single increment under $\widehat{\Psi}^\indep$
(recall \eqref{eq:PSihatindep}).
\begin{lemma}
\label{lem:1djointCLTXhat}
There exist $C>0$, $b \in (0,1/4)$ such that for all bounded
Lipschitz continuous $f :\R^2 \to \R$
\begin{align}
\left| \E\Big[ f\big(\tfrac{\widehat X_n}{\widehat{\sigma}\sqrt{n}},
\tfrac{\widehat X'_n}{\widehat{\sigma}\sqrt{n}}\big) \Big]
- \E\big[ f(Z) \big] \right| \le L_f \frac{C}{n^b}
\end{align}
where $Z$ is two-dimensional standard normal and $L_f$ the Lipschitz
constant of $f$.
\end{lemma}

By Lemma~\ref{lemma:couplingjointregen}, there exist
$C_1, K >0$ such that for $x, x' \in \Z$ with $|x-x'| \ge K \log n$
\begin{align}
\label{eq:phi1bd1}
\big| \phi_1(x,x') \big|, \big| \phi_2(x,x') \big|,
\big| \phi_{12}(x,x') \big| & \le \frac{C_1}{n^2}, \\
\label{eq:phi11bd1}
\big| \phi_{11}(x,x') - \widehat{\sigma}^2 \big| ,
\big| \phi_{22}(x,x') - \widehat{\sigma}^2 \big| & \le \frac{C_1}{n^2}.
\end{align}
Put
\begin{equation}
R_n \coloneqq \#
\big\{ 0 \le j \le n : | \widehat X_j - \widehat X_j' | \le K \log n \big\}.
\end{equation}
We expect that $R_n =o(n)$ in probability (see below for a proof),
which combined with \eqref{eq:phi1bd1}, \eqref{eq:phi11bd1} would yield
\begin{align}
\label{eq:1dquadvar:prelimclaim1}
& \frac{A_n^{(11)}}{n} \to \widehat{\sigma}^2, \quad
\frac{A_n^{(22)}}{n} \to \widehat{\sigma}^2, \quad
\frac{A_n^{(12)}}{n} \to 0 \,  \qquad \text{in $\mathbb{P}(\cdot | B_0)$-probability as $n\to\infty$}.
\end{align}
Note that \eqref{eq:1dquadvar:prelimclaim1} together with our (exponential)
tail bounds for the differences would already imply a two-dimensional CLT
for $(M_n/\sqrt{n}, M_n'/\sqrt{n})$ by standard martingale CLT results.
Since we require quantitative control in the martingale CLT,
we have to estimate a little more carefully.
\smallskip

For $n \in \N$, let
\begin{equation}
Q_n \coloneqq \left(\begin{array}{cc}
\phi_{11}(\widehat X_{n-1},\widehat X_{n-1}') &
\phi_{12}(\widehat X_{n-1},\widehat X_{n-1}') \\
\phi_{12}(\widehat X_{n-1},\widehat X_{n-1}') &
\phi_{22}(\widehat X_{n-1},\widehat X_{n-1}')
\end{array}\right)
\end{equation}
be the conditional covariance matrix given $\widehat{\mathcal{F}}_{n-1}$
of the $\R^2$-valued random variable $(M_n-M_{n-1}, M'_n-M'_{n-1})$,
let $\lambda_{n,1} \ge \lambda_{n,2} \ge 0$ be its
eigenvalues. We obtain from \eqref{eq:phi1bd1}, \eqref{eq:phi11bd1}
and \eqref{eq:phiunifbd} together with well-known stability properties
for the eigenvalues of symmetric matrices that
\begin{align}
\big| \lambda_{j+1,1} - \widehat{\sigma}^2 \big| +
\big| \lambda_{j+1,2} - \widehat{\sigma}^2 \big|
\le C_2 \ind{| \widehat X_j - \widehat X_j' | \le K \log n}
+ \frac{C_2}{n^2} \ind{| \widehat X_j - \widehat X_j' | > K \log n}
\end{align}
for some $C_2 < \infty$, see, e.g., \cite[\S41]{MR0158519}
or \cite[Ch.~IV, Thm.~3.1]{StewartSun1990}. In particular,
\begin{align}
\label{eq:Rackauskas.type.bd}
\sum_{i=1}^2
\left| n \widehat{\sigma}^2 - \sum\nolimits_{j=1}^n \lambda_{j,i} \right|
\le \frac{C_2}n + C_2 R_n.
\end{align}

\begin{lemma}
\label{lem:RnandCnbds}
\begin{enumerate}
\item There exist $0 \le \delta_R < 1/2$, $c_R < \infty$ such that
\begin{align}
\label{eq:Rn.moment.bd}
\E\big[ R_n^{3/2} \big] \le c_R n^{1+\delta_R} \quad \text{for all}\; n.
\end{align}
\item There exist $\delta_C > 0$, $c_C < \infty$ such that
\begin{align}
\label{eq:Cnsmall}
\E\left[ \frac{|A_n^{(1)}|}{\sqrt{n}}\right], \,
\E\left[ \frac{|A_n^{(2)}|}{\sqrt{n}}\right] \le \frac{c_C}{n^{\delta_C}}
\qquad \text{for all}\; n.
\end{align}
\end{enumerate}
\end{lemma}

Now we have all ingredients for the
\begin{proof}[Proof of Lemma~\ref{lem:1djointCLTXhat}]
Let $f :\R^2 \to \R$ be bounded Lipschitz continuous with
Lipschitz constant $L_f$ and $Z$ two-dimensional standard normal.
We obtain from \eqref{eq:Rackauskas.type.bd},
\eqref{eq:Rn.moment.bd} and Corollary~1.3 in \cite{Rackauskas:1995}
that
\begin{align}
\label{eq:MnRackauskasbound}
\left| \E\Big[ f\big(\tfrac{M_n}{\widehat{\sigma}\sqrt{n}},
\tfrac{M'_n}{\widehat{\sigma}\sqrt{n}}\big) \Big]
- \E\big[ f(Z) \big] \right| \le L_f \frac{C}{n^{b'}} \qquad
\text{for all}\; n
\end{align}
for some $C < \infty$ and $b'=(\tfrac{1}{2}-\delta_R)/3$.
(Read $X_k=\big( (M_k-M_{k-1})/\sqrt{\widehat{\sigma}^2 n},
(M'_k-M'_{k-1})/\sqrt{\widehat{\sigma}^2 n}\big)$ in
\cite[Cor.~1.3]{Rackauskas:1995}, note that
$\sup_{k} \E\big[\norm{(M_k-M_{k-1}, M'_k-M'_{k-1})}^3\big] < \infty$
by the uniform exponential tail bounds from
Lemma~\ref{lemma:jointrenewaltails}.)

Combining \eqref{eq:MnRackauskasbound} and \eqref{eq:Cnsmall}
yields
\begin{align}
\left| \E\Big[ f\big(\tfrac{\widehat X_n}{\widehat{\sigma}\sqrt{n}},
\tfrac{\widehat X'_n}{\widehat{\sigma}\sqrt{n}}\big) \Big]
- \E\big[ f(Z) \big] \right|
\le L_f \frac{C}{n^{b'}} + L_f  \frac{c_C}{\widehat{\sigma} n^{\delta_C}}.
\end{align}
\end{proof}
\medskip

To prepare the proof of Lemma~\ref{lem:RnandCnbds} we need some further
notation: Put (with suitable $b \in (0,1/2)$ and $K \gg 1$, see below) for
$n\in\N$ $\mathcal{R}_{n,0}\coloneqq0$ and for $i \in \N$
\begin{align*}
  \mathcal{D}_{n,i} & = \min\{ m > \mathcal{R}_{n,i-1} : \abs{\widehat X_m - \widehat X_m'} \ge n^b\},\\
  \mathcal{R}_{n,i} & = \min\{ m >\mathcal{D}_{n,i}: \abs{\widehat X_m - \widehat X_m'} \le K \log n \},
\end{align*}
then $[\mathcal{R}_{n,i-1}, \mathcal{D}_{n,i})$ is the $i$th ``black box interval'' on
``coarse-graining level'' $n$ (note that for $m \in \cup_i [\mathcal{R}_{n,i-1}, \mathcal{D}_{n,i})$
the coupling result, Lemma~\ref{lemma:couplingjointregen}, does not help;
whereas for $m \in \cup_i [\mathcal{D}_{n,i}, \mathcal{R}_{n,i})$, we can couple
$(\widehat{X}_m, \widehat{X}'_m)$ with a pair of walks on independent
copies of the cluster up to a small error term).

We distinguish four possible types of such black box intervals,
depending on the ordering of $\widehat X$ and $\widehat X'$ at the
beginning and end of the interval: Set
\begin{equation}
\label{eq:Widef}
W_{n,i} \coloneqq \begin{cases}
1 & \text{if} \; \; \widehat{X}_{\mathcal{R}_{n,i-1}} > \widehat{X}'_{\mathcal{R}_{n,i-1}},
\widehat{X}_{\mathcal{D}_{n,i}} < \widehat{X}'_{\mathcal{D}_{n,i}}, \\
2 & \text{if} \; \; \widehat{X}_{\mathcal{R}_{n,i-1}} > \widehat{X}'_{\mathcal{R}_{n,i-1}},
\widehat{X}_{\mathcal{D}_{n,i}} > \widehat{X}'_{\mathcal{D}_{n,i}}, \\
3 & \text{if} \; \; \widehat{X}_{\mathcal{R}_{n,i-1}} < \widehat{X}'_{\mathcal{R}_{n,i-1}},
\widehat{X}_{\mathcal{D}_{n,i}} > \widehat{X}'_{\mathcal{D}_{n,i}}, \\
4 & \text{if} \; \; \widehat{X}_{\mathcal{R}_{n,i-1}} < \widehat{X}'_{\mathcal{R}_{n,i-1}},
\widehat{X}_{\mathcal{D}_{n,i}} < \widehat{X}'_{\mathcal{D}_{n,i}}.
\end{cases}
\end{equation}
By construction and the strong Markov property of
$(\widehat X_m, \widehat X_m')_m$ [with a grain of salt because
at a time $\mathcal{R}_{n,i}$, the difference may be $[K \log n]$ or
$[K \log n]-1$, etc.] we have the following:
For each $n\in\N$,
\begin{align}
& (\mathcal{R}_{n,i}-\mathcal{D}_{n,i})_{i=1,2,\dots} \;\; \text{is an i.i.d.\ sequence, and} \\
& (W_{n,i},\mathcal{D}_{n,i}-\mathcal{R}_{n,i-1})_{i=2,3,\dots} \;\; \text{is a Markov chain};
\end{align}
furthermore, the two objects are independent, the transition probabilities of
$(W_{n,i},\mathcal{D}_{n,i}-\mathcal{R}_{n,i-1})_{i}$ depend only on the first (the ``type'') coordinate,
and a bound of the form analogous to Lemma~\ref{lemma:separation1}
holds.
\begin{lemma}[Separation and overcrossing lemma for $d=1$]
\label{lem:seplemd=1}
We can choose $0<b_2<1/4$, $b_3, b_4 >0$ such that
\begin{align}
\label{eq:TiminSibd1}
\Pr\big( \mathcal{D}_{n,i}-\mathcal{R}_{n,i-1} \geq n^{b_2} \, \big| \,
W_{n,i}=w \big)
\leq e^{-b_3 n^{b_4}}, \quad w \in \{1,2,3,4\}, \; n \in \N.
\end{align}
Furthermore, there exists $\epsilon>0$ such that uniformly in $n$
\begin{align}
\label{eq:Wntransbd}
\Pr(W_{n,2}= a' \, | \, W_{n,1}= a) \ge \epsilon
\end{align}
for all pairs of types $(a,a') \in \{1,2,3,4\}^2$ where a transition is
``logically possible'' (cf.~\eqref{eq:Widef}).
\end{lemma}
\begin{proof}[Proof sketch]
The proof of \eqref{eq:TiminSibd1} is analogous to that of
Lemma~\ref{lemma:separation1}, making use of the $d=1$-case
of Lemma~\ref{lemma:harmonicballsd3}.
\smallskip

For \eqref{eq:Wntransbd} the crucial point is to show that
when $\widehat{X}$ and $\widehat{X}'$ have come closer than
$K \log n$ at time $m=\mathcal{R}_{n,i}$ with $\widehat{X}_m > \widehat{X}'_m$, say,
there is a chance of at least $\delta > 0$ (uniformly in $n$)
that they reverse their roles before reaching a distance of $n^b$,
i.e.\ there is $j < \mathcal{D}_{n,i+1}$ such that $\widehat{X}_j \le \widehat{X}'_j$.

To see this, write $D_j\coloneqq\widehat{X}_j - \widehat{X}'_j$, pick
$0 < \varepsilon \ll K$ (to be tuned later). When the process $D$
starts from $K \log n$, there is a chance $\ge \delta' > 0$ that it
reaches $\varepsilon \log n$ before $2K \log n$ within less than
$\log^3 n$ steps (use the coupling from
Lemma~\ref{lemma:couplingjointregen} and analogous results for simple
random walk on $\Z^1$, note that the probability that the coupling
fails is at most $Ce^{-c \varepsilon n} \log^3 n$); once $D_j \le
\varepsilon \log n$ there is a chance of at least $\exp(-c'
\varepsilon \log n) = n^{-c' \varepsilon}$ that $D$ hits $(-\infty,0]$
with the next $(\varepsilon \log n)/2$ steps (construct suitable
``corridors'' as in Step~1 of the proof of
Lemma~\ref{lemma:separation1}). When $D$ does not hit $(-\infty,0]$
but instead reaches $2K \log n$ observe that
\[
\Pr\big( \text{$D$ hits $K \log n$ before $n^b$} \, \big| \,
D_0 = 2K \log n \big) \approx
\frac{n^b- 2K \log n}{n^b- K \log n} = 1 - \frac{K \log n}{n^b- K \log n}.
\]
By the Markov property, we will thus have a geometric number
of excursions from $2 K \log n$ that reach $K \log n$ but not $n^b$,
and each of these has a chance $\ge \delta' n^{-c' \varepsilon}$
to hit $(-\infty,0]$. Thus, if $\varepsilon$ is chosen
so small that $c' \varepsilon < b$ there is a substantial chance that
one of them will be successful.
\end{proof}

Note that \eqref{eq:Wntransbd} guarantees that the chain $(W_{n,i})_{i}$ is
uniformly in $n$ (exponentially) mixing. By symmetry of the
construction,
\begin{align}
\label{eq:Wnjsymmetries}
\Pr(W_{n,j}=1) = \Pr(W_{n,j}=3) \quad \text{and} \quad
\Pr(W_{n,j}=2) = \Pr(W_{n,j}=4) \quad \text{for all}\; j, n,
\end{align}
and the same holds for the stationary distribution $\pi_{W,n}$ of
$(W_{n,j})_j$.

\begin{proof}[Proof of Lemma~\ref{lem:RnandCnbds}]
Let $Y$ have distribution
\[
\Pr(Y \ge \ell) = \widehat{\Psi}^\indep\Big( \inf\{ m \ge 0 : \widehat{X}_m <
\widehat{X}' _m\} \ge \ell \, \Big| \,
(\widehat{X}_0, \widehat{X}'_0)=(1,0) \Big), \quad\ell \in \N
\]
and let $V$ be an independent, Bernoulli($1-1/n$)-distributed random
variable. A simple coupling construction based on
Lemma~\ref{lemma:couplingjointregen} shows that
$\mathcal{R}_{n,1}-\mathcal{D}_{n,1}$ is stochastically larger than
\begin{equation}
\label{eq:goodtimesbound}
\big( (1-V) + VY \big) \wedge n
\end{equation}
(by choosing $K$ appropriately, we can ensure that the probability
that the relevant coupling between $\widehat{\Psi}^\indep$ and
$\widehat{\Psi}^\joint$ fails during the first $n$ steps is less than
$1/n$). In particular, by well-known tail probability estimates for
ladder times of one-dimensional random walks, there exist $c>0$ and
$c_Y >0$ such that uniformly in $n$,
\begin{align}
\label{eq:YprimeLaplaceTf}
\E\big[ e^{-\lambda ((1-V) + VY)} \big] & \le \exp\big(-c_Y \sqrt{\lambda} \big),
\quad \lambda \ge 0 \qquad \text{and} \\[1ex]
\label{eq:S1-T1tails}
\Pr(\mathcal{R}_{n,1}-\mathcal{D}_{n,1} \ge \ell) & \ge \frac{c}{\sqrt{\ell}},
\quad \ell=1,\dots,n.
\end{align}

Let $I_n\coloneqq\max\{ i : T_i \le n\}$ be the number of ``black boxes'' that
we see up to time $n$. We have $I_n = \mathcal O(\sqrt n)$ in probability and in
fact
\begin{align}
\label{eq:Inbd2}
\E[I_n^2] & \le C n
\end{align}
as can be seen from \eqref{eq:S1-T1tails}
by comparison with a renewal process with inter-arrival
law given by the return times of a (fixed) one-dimensional random
walk.
\smallskip

More quantitatively, there is $c>0$ such that for $1 \le k \le n$,
\begin{align}
\label{eq:Intailbd}
\Pr(I_n \ge k) \le \exp\big( - ck^2/n \big),
\end{align}
so in particular
\begin{align}
\label{eq:Inuibound}
\E\big[ I_n \ind{I_n \ge n^{3/4}}\big]
= \sum_{k=\lceil n^{3/4} \rceil}^n \Pr(I_n \ge k)
\le n e^{-c\sqrt{n}}.
\end{align}
{\small (\eqref{eq:Intailbd} is a standard result for lower deviations of
  a heavy-tailed renewal process. For completeness' sake (and lack of
  a point reference), here are some details: Let $Y'_1, Y'_2, \dots$
  be i.i.d.\ copies of $\big( (1-V) + VY \big)$ from
  \eqref{eq:goodtimesbound}, then for $1 \le k \le n$
\begin{align}
 \Pr(I_n \ge k) = & \,
\Pr\big( (Y'_1 \wedge n) + \cdots + (Y'_k \wedge n) \le n \big)
= \Pr\big(Y'_1 + \cdots + Y'_k \le n\big) \notag \\
\le & \,
e^{\lambda n} \Big(\E\big[\exp(-\lambda Y_1)\big]\Big)^k
\le \exp\big( \lambda n - k c_Y \sqrt{\lambda} \big)
\end{align}
for any $\lambda > 0$ by \eqref{eq:YprimeLaplaceTf},
now put $\lambda \coloneqq (c_Yk/n)^2$.)}
\smallskip

Note that
\begin{align}
\label{eq:Rnbd1}
R_n \le \sum_{j=1}^{I_n+1} (\mathcal{D}_{n,j}-\mathcal{R}_{n,j-1}),
\end{align}
thus, using \eqref{eq:TiminSibd1}, indeed $R_n = o(n)$ in probability, and
\eqref{eq:Inbd2} together with \eqref{eq:Rnbd1},
\eqref{eq:TiminSibd1} implies
\eqref{eq:Rn.moment.bd}:
\[
\E\big[R_n^2\big] \le n^2 \Pr\big( \exists \, j \le n :
\mathcal{D}_{n,j+1}-\mathcal{R}_{n,j} \ge n^{b_2} \big) + n^{2 b_2} \E[I_{n+1}^2] \le C n^{1+2b_2},
\]
and $\E\big[R_n^{3/2}\big] \le \big( \E\big[R_n^2\big] \big)^{3/4}$.
\medskip

For \eqref{eq:Cnsmall} we must make use of
cancellations
in the increments of $A_n^{(1)}, A_n^{(2)}$, making use of the
fact that ``opposite'' types of crossings of $\widehat{X}$ and
$\widehat{X}'$ appear asymptotically with the same frequency. Let
\[
  D_{n,m} \coloneqq
  A_{\mathcal{D}_{n,m}}^{(1)}-A_{\mathcal{R}_{n,m-1}}^{(1)},
  \quad
  D'_{n,m} \coloneqq
  A_{\mathcal{D}_{n,m}}^{(2)}-A_{\mathcal{R}_{n,m-1}}^{(2)}.
\]
By symmetry,
\begin{align}
\label{eq:Dnjsymmetries}
\E\big[ D_{n,j} \big] = 0, \quad
\E\big[ D_{n,j} \, \big| \, W_{n,j}=1\big] & =
-\E\big[ D_{n,j} \, \big| \, W_{n,j}=3\big],
\notag \\[0.3ex]
\E\big[ D_{n,j} \, \big| \, W_{n,j}=2\big] & =
-\E\big[ D_{n,j} \, \big| \, W_{n,j}=4\big],
\end{align}
\eqref{eq:TiminSibd1} and \eqref{eq:phiunifbd} together show that
for some $C < \infty$, $b>0$ uniformly in $j, n \in \N$, $w \in \{1,2,3,4\}$
\begin{align}
\label{eq:Dnjabsbound}
\E\big[ |D_{n,j}| \, \big| \, W_{n,j}=w\big] \le C n^b
\end{align}
and analogously for $D_{n,j}'$.
Put $\mathcal{G}_j \coloneqq \widehat{\mathcal{F}}_{\mathcal{D}_{n,j}}$ (the
$\sigma$-field of the $\mathcal{D}_{n,j}$-past) for $j \in \N$, for $j \le 0$
let $\mathcal{G}_j$ be the trivial $\sigma$-algebra.
$D_{n,j}, D_{n,j}'$ are $\mathcal{G}_j$-adapted for $j \in \N$.
Since for $k < m$
\[
\E\left[ \left. D_{n,m} \right| \mathcal{G}_k \right]
= \E\big[ \E[D_{n,m} \, | W_{n,m}] \, \big| \, \mathcal{G}_k \big]
\]
by construction and $(W_{n,j})_j$ is (uniformly in $n$) exponentially
mixing we have, observing \eqref{eq:Dnjsymmetries},
\eqref{eq:Wnjsymmetries}, and \eqref{eq:Dnjabsbound}
\begin{align}
\E\left[ \big( \E\left[ \left. D_{n,m} \right| \mathcal{G}_{m-j} \right] \big)^2 \right]
\le C e^{-cj}, \quad m, j \in \N, \; n \in \N
\end{align}
for some $C,c \in (0,\infty)$ and analogous bounds for $D_{n,m}'$.
Let $S_{n,m} \coloneqq \sum_{j=1}^m D_{n,j}$, $S'_{n,m} \coloneqq \sum_{j=1}^m D'_{n,j}$ then
for each $n \in \N$, $(S_{n,m})_m$ is a mixingale with uniformly (in $n$)
controlled mixing rate (see \cite{HallHeyde:1980}, p.~19).

Thus, using McLeish's analogue of Doobs $\mathcal{L}^2$-inequality for
mixingales (e.g.\ \cite{HallHeyde:1980}, Lemma~2.1), we have
\begin{align}
\label{eq:Doobformixingale}
\E\left[ \max_{m=1,\dots,n^{3/4}} S_{n,m}^2 \right] \le K n^{3/4},
\end{align}
hence
\begin{align}
\label{eq:SnInfluctbd}
\E\left[ \frac{|S_{n,I_n}|}{\sqrt{n}} \ind{I_n \le n^{3/4}} \right]
\le \frac{1}{\sqrt{n}}
\left( \E\left[ \max_{m=1,\dots,n^{3/4}} S_{n,m}^2 \right]\right)^{1/2}
\le \frac{K^{1/2}}{n^{1/8}}.
\end{align}
By (\ref{eq:phi1bd1}) we have
$A_n^{(1)} = \sum_{j=1}^{I_n} D_{n,j} + O(1)$, so
\begin{align}
  \E\left[ \frac{|A_n^{(1)}|}{\sqrt{n}}\right]
\le \frac{c}{\sqrt{n}} & +
\frac{1}{\sqrt{n}}
\E\left[ \frac{|S_{n,I_n}|}{\sqrt{n}} \ind{I_n \le n^{3/4}} \right] +
\frac{1}{\sqrt{n}} C n^{1+b_2}
\E\big[ I_n \ind{I_n \ge n^{3/4}}\big] \notag \\
\label{eq:CnL1bd}
& {} + \frac{1}{\sqrt{n}} cn
\Pr\big(\exists \, j \le n \, : \, \mathcal{D}_{n,j}-\mathcal{R}_{n,j-1} \ge n^{b_2} \big).
\end{align}
Using \eqref{eq:SnInfluctbd}, \eqref{eq:Inuibound} and
\eqref{eq:TiminSibd1}, respectively on the last three terms on the
right hand side (and analogous bounds for $A_n^{(2)}$) yields \eqref{eq:Cnsmall}.
\end{proof}

\begin{remark}
The arguments used in the proof of Lemma~\ref{lem:RnandCnbds} can
be used to show that there exists $C < \infty$ such that for all $n \in \N$
\begin{align}
\label{eq:DoobtypeboundforXhat}
\sup_{\ell \in \N_0} \E\left[ \sup_{0 \le k \le n^{2/3}}
\left( \widehat{X}_{\ell+k} - \widehat{X}_{\ell} \right)^2 \right] \leq
C n^{2/3}
\end{align}
and analogously for $\widehat{X}'$.
\end{remark}
\begin{proof}[Sketch of proof]
  Decompose $\widehat{X}_{n} = M_n + A_n^{(1)}$ (recall \eqref{defMnMnhat}).
  The analogue of \eqref{eq:DoobtypeboundforXhat} for the martingale $(M_n)$
  holds by Doob's $\mathcal{L}^2$-inequality (note that by the uniform exponential tail
  bounds we have $\sup_{n\in\N} \E\left[(M_n-M_{n-1})^2\right] < \infty$).

  To obtain the analogue of \eqref{eq:DoobtypeboundforXhat} for the process
  $(A_n^{(1)})$ note that by the Markov property of $(\widehat{X}, \widehat{X}')$
  it suffices to verify that uniformly in $x_0,x_0' \in \Z$
  \begin{align}
    \E\left[ \left. \sup_{0 \le k \le n^{2/3}} \left( A_{k}^{(1)} -
        A_{0}^{(1)} \right)^2
      \, \right| \, (\widehat{X}_0, \widehat{X}'_0)=(x_0,x_0') \right] \leq
    C n^{2/3}
  \end{align}
  for some $C<\infty$.
  This can be proved by expressing $A_k^{(1)}$ as a sum of mixingale increments as
  in the proof of Lemma~\ref{lem:RnandCnbds}
  and suitably adapting the argument around
  \eqref{eq:Doobformixingale}--\eqref{eq:CnL1bd}.
\end{proof}

\subsubsection{Transferring from $\widehat{X}$ to $\widetilde{X}$
and completion of the proof \\of
Lemma~\ref{lemma:qCLTars},
for the case $d=1$}

Note that Lemma~\ref{lem:1djointCLTXhat} gives almost the required result except
that it speaks about $(\widehat{X}, \widehat{X}')$, a pair of walks observed
along \emph{joint} regeneration times, rather than two walks each observed along
its individual sequence of regeneration times. Here, we indicate how to remedy
this. \smallskip

Let (with regeneration times $T_m$, $T_m'$ as in (\ref{eq:defRi}--\ref{eq:defRi'})
and $T^{\Sim}_m$ from \eqref{eq:12})
\begin{align*}
   L_n & \coloneqq \max\bigl\{ m \le n: T_m \in \{T_0',T_1',\dots\} \bigr\}, \\
   L_n' & \coloneqq \max\bigl\{ m \le n: T_m' \in \{T_0,T_1,\dots\} \bigr\}, \\
  \widehat L_n & \coloneqq \max\bigl\{ m \le n: T_m^{\Sim} \le
 T_n \bigr\},  \\
  \widehat L_n' & \coloneqq \max\bigl\{ m \le n: T_m^{\Sim} \le
 T_n' \bigr\},
\end{align*}
i.e., $L_n$ and $L_n'$ are the indices of the last joint regeneration time
before the $n$-th with respect to the walks $X$ respectively  $X'$ and
$\widehat L_n$ respectively $\widehat L_n'$ is the corresponding
number of joint regeneration times,
in particular
\begin{align}
  \label{eq:2}
  \bigl(\widetilde{X}_{L_n}, T_{L_n}\bigr) =
  \bigl(\widehat{X}_{\widehat L_n}, T^{\Sim}_{\widehat L_n}\bigr)
  \quad \text{and} \quad
  \bigl(\widetilde{X}'_{L_n'}, T'_{L_n'}\bigr) =
  \bigl(\widehat{X}'_{\widehat L_n'}, T^{\Sim}_{\widehat L_n'}\bigr) .
\end{align}

\begin{lemma}
\label{lem:jointrenewalcounts}
There exist $C < \infty$, $q \in (0,1]$ such that for all $n \in \N$
\begin{align}
\label{eq:jointrenewalcounts1}
\E^\joint_{\mathbf 0, \mathbf 0}\left[ \left( n- L_n \right)^2 \right] \leq C,
\end{align}
\begin{align}
\label{eq:jointrenewalcounts2}
\E^\joint_{\mathbf 0, \mathbf0}\left[ \left( \widehat{L}_n - nq \right)^2 \right] \leq C n,
\end{align}
and analogously for $L_n'$ and $\widehat{L}_n'$.
\end{lemma}
\begin{proof}[Sketch of proof]
Note that under $\Psi^\indep$, when the two walks use independent copies of the
cluster, $(T_n)$ and $(T_n')$ are two independent renewal processes
(whose waiting times have exponential tail bounds), hence
$T_n/n \to \mu$, $T^\Sim_n/n \to \widehat{\mu}$ {a.s.\ and in $\mathcal{L}^2$}
with some $0< \mu < \widehat{\mu} < \infty$, and
\eqref{eq:jointrenewalcounts1}, \eqref{eq:jointrenewalcounts2} hold
(with $q=\mu/\widehat{\mu}$).

By suitably ``enriching'' the coupling arguments used above, i.e.\ using
$\Psi^\joint$ instead of $\widehat{\Psi}^\joint$, and then
reading off the number of individual renewals between joint renewals we
see that \eqref{eq:jointrenewalcounts1}, \eqref{eq:jointrenewalcounts2} also
hold for two walks on the same cluster.
\end{proof}

Now write
\begin{align}
  \label{eq:proofquenchedCLTd=1.1}
  (\widetilde X_n, \widetilde X_n') =
  \bigl(\widehat X_{[nq]}, \widehat X_{[nq]}' \bigr)
  +  \bigl(\widehat X_{\widehat L_n} - \widehat X_{[nq]},
  \widehat X_{\widehat L_n'}' - \widehat X_{[nq]}' \bigr)
  +  \bigl(\widetilde X_n-\widetilde X_{L_n},
  \widetilde X_n'-\widetilde X_{L_n'}' \bigr),
\end{align}
hence
\begin{align}
  \label{eq:jointtoindrenewal}
  & \E\left[ \bigl| \widetilde X_n - \widehat X_{[nq]} \bigr| +
    \bigl| \widetilde X_n' - \widehat X_{[nq]}' \bigr| \right] \notag \\
  & \leq
  \E\left[ \bigl| \widehat X_{\widehat L_n} - \widehat X_{[nq]} \bigr|
    + \bigl| \widehat X_{\widehat L_n'}' - \widehat X_{[nq]}' \bigr| \right]
  + \E\left[ \bigl| \widetilde X_n-\widetilde X_{L_n}\bigr|
    + \bigl| \widetilde X_n'-\widetilde X_{L_n'}' \bigr| \right]
  \leq \frac{C}{n^b}
\end{align}
for some $C<\infty$, $b > 0$, using Lemma~\ref{lem:jointrenewalcounts}
and the fact that the increments of $\widetilde X$, $\widetilde X'$,
$\widehat X$, and $\widehat X'$ have uniformly exponentially bounded
tails, and Chebyshev's inequality.
\eqref{eq:jointtoindrenewal} together with Lemma~\ref{lem:1djointCLTXhat}
implies Lemma~\ref{lemma:qCLTars} for $d=1$.
\hfill \qed

\begin{appendix}

\section{An auxiliary result}
\label{sect:auxproofs}

Recall the definition of $\tau^{\mathbf{0}}$ in (and after) equation
\eqref{def:tauA}. The following result is ``folklore'' but we did not
find a suitable reference (only for the corresponding contact process
version of the result or for a special case; see
Remark~\ref{rem:known-result} below).
\begin{lemma} \label{lem:exp-bound-perc}
  For $p>p_c$ there exist $C,\gamma \in (0,\infty)$ such that
  \begin{align}
    \label{est:exp-bound-perc}
    P_p(n \leq \tau^{\mathbf 0} < \infty) \leq Ce^{-\gamma n}.
  \end{align}
\end{lemma}
\begin{remark}
  \label{rem:known-result}
  The above result is proven in \cite{Durrett:84} for the ``conventional
  oriented percolation'' on $\Z \times \Z_+$ with critical value $p_c^{(1)}$.
  Dominating the ``conventional oriented percolation'' it is easy
  to see that the above result is true for any $d\ge 1$ and $p >p_c^{(1)}$
  where $p_c^{(1)}$ is the critical value for the ``conventional oriented
  percolation'' on $\Z \times \Z_+$ considered in \cite{Durrett:84}. It is also
  clear that for $d\ge 1$ we have $p_c \leq p_c^{(1)}$. Our task here is to
  extend the result to $p \in (p_c, p_c^{(1)}]$.
\end{remark}

\begin{proof}
  We will adapt arguments from \citep[][p.~57-58]{Liggett:99}, where this result
  was proven for the contact process.

  According to \citep[][p.~7 in arXiv version]{GrimmetHiemer:02} there exist $r$
  (large enough) such that
  \begin{align*}
    \Pr (\tau^{[-r,r]^d} < \infty) <  \varepsilon.
  \end{align*}
  Furthermore, by standard arguments, on a coarse grained grid one can construct
  a percolation structure with probability of open sites
  $p_{\textnormal{coarse}} > p_c^{(1)}$, such that it is dominated by the
  process of suitably defined space-time blocks of the original percolation
  \citep[see][]{GrimmetHiemer:02}, where the blocks are such that on the
  ``bottom'' of the block all sites in some space-time translate of
  $[-r,r]^d\times\{0\}$ are open.

  Now the idea is to show that for large $n$ on the event $\{\tau^{A}>n\} $ it
  is likely that the ``domination'' described above has started.

  We set
  \begin{align*}
    \delta\coloneqq \Pr(\eta_r^{0}=[-r,r]^d).
  \end{align*}
  Now we define a random variable $N$ (measurable w.r.t.\ to $\sigma(\Omega_p)$)
  such that
  \begin{align*}
    \Pr(N=k) = \delta(1-\delta)^k, \; k \ge 0,
  \end{align*}
  and
  \begin{align*}
    \text{either } \eta_{Nr}^{0} = \emptyset, \; \text{ or } \; x + [-r,r]^d
    \subset \eta_{(N+1)r}^{0} \text{ for some } x \in \Z^d.
  \end{align*}
  Set $N=0$ if $\eta_r^0=[-r,r]^d$. If $\eta_r^{0} \ne [-r,r]^d$, i.e.\ on $\{N
  >0\}$ we have either $\eta_r^{0}=\emptyset$ or $\eta_r^{0}\ne \emptyset$. In
  the first case $\eta_n^{0} =\emptyset$ for all $n \ge r$ and therefore
  $N\ge1$. In the second case, restart (a subprocess) in some $x \in
  \eta_r^{0}$, and let $N=1$ if
  \begin{align*}
    {}_r \eta_{2r}^{x} \subset x + [-r,r]^d.
  \end{align*}
  Here, for $m\le n$ we use ${}_m \eta_{n}^{x}$ to denote the discrete time
  contact process at time $n$ starting in $\{x\}$ at time $m$. Again on the
  complement $\{N=1\}$ either ${}_r \eta_{2r}^{x} = \emptyset$, in which case $N
  \ge 2$, or ${}_r \eta_{2r}^{x} \ne \emptyset$, in which case we can proceed as
  before.

  If $x +[-r,r]^d \subset \eta_{(N+1)r}^{0}$ for some $x \in \Z^d$ then we can
  start the coupling with percolation on a coarse grained grid described at the
  beginning of the proof.

  Let $(N+1)r+M$ be the extinction time of this block percolation process. Note
  that $(N+1)r$ is the time at which the comparison starts and $M$ is the
  extinction time of the discrete contact process with probability of open sites
  given by $p_{\textnormal{coarse}}$.

  As noted in Remark~\ref{rem:known-result} we have $\Pr(M
  >n|M<\infty) \le C e^{-\gamma n}$ for suitable $C,\gamma >0$. If
  $M=\infty$ then $\tau^{\mathbf 0}=\infty$. If $M<\infty$ then at
  time $M+(N+1)r$ the configuration $\eta_{M+(N+1)r}$ is empty or not.
  If it is not empty then we repeat the procedure and obtain an i.i.d.
  sequence of independent random variables $N_i$ with the same law as
  $N$ and independent random variables $M_i$ with the same law as $M$
  conditioned on $M<\infty$. Let $L$ be such that at time
  \begin{align*}
    \sigma = \sum^{L}_{i=1} ((N_i+1)r+M_i)
  \end{align*}
  either $\eta_{\sigma}^0 =\emptyset$ or $\tau^{\mathbf 0}=\infty$. Thus, $\sigma >
  \tau^{\mathbf 0}$ on $\{\tau^{\mathbf 0} < \infty\}$ and we obtain
  \begin{align*}
    \Pr(n < \tau^{\mathbf 0} < \infty) \le \Pr (\sigma > n) \le C e^{-\gamma n}
\end{align*}
for suitable $C,\gamma>0$.
\end{proof}

\begin{remark}
  \label{rem:exp-bound-perc-unif}
  Inspection of the proof of Lemma~\ref{lem:exp-bound-perc}
  shows that the constants $\gamma$ and $C$ in
  \eqref{est:exp-bound-perc} can be chosen to apply uniformly
  in $p \in [p_c+\delta,1]$ for any $\delta>0$.
\end{remark}

\end{appendix}

\subsection*{Acknowledgements}
We would like to thank two anonymous referees for their careful
comments which helped to improve the presentation of the paper.
\noindent
Support by Deutsche Forschungsgemeinschaft
through DFG grant BI\ 1058/3-1 (M.B.), DFG grant GA\ 582/7-1 (N.G.)
and DFG grant Pf\ 672/6-1 (A.D.) is gratefully acknowledged.

\bibliographystyle{chicago}

\vspace{2.3cm}

{ \footnotesize \parindent 0pt \parskip 3mm

  Matthias Birkner:
  Johannes Gutenberg-Universit\"at Mainz,
  Institut f\"ur Mathematik,
  Fachbereich 08 -- Mathematik, Physik und Informatik,
  Staudingerweg~9,
  55099 Mainz,
  Germany.

  Ji\v r\'i \v Cern\' y:
  University of Vienna,
  Faculty of Mathematics,
  Nordbergstrasse~15,
  1090 Vienna,
  Austria.

  Andrej Depperschmidt:
  Abteilung f\"ur Mathematische Stochastik,
  Universit\"at Freiburg,
  Eckerstr.~1,
  79104 Freiburg i.\ Br.,
  Germany.

  Nina Gantert:
  Technische Universit\"at M\"unchen,
  Fakult\"at f\"ur Mathematik,
  Boltzmannstra\ss e~3,
  85748 Garching bei M\"unchen,
  Germany.
}

\end{document}